\documentclass[12pt]{article}

\usepackage{stmaryrd}
\usepackage{mathrsfs}
\usepackage{amsmath,amsthm,amssymb}
\usepackage[colorlinks,
            linkcolor=red,
            anchorcolor=blue,
            citecolor=magenta
            ]{hyperref}
\usepackage{pdfsync}
\usepackage[numbers,sort&compress]{natbib}
\allowdisplaybreaks

\theoremstyle{definition}

\theoremstyle{plain}
\newtheorem{theorem}{Theorem}[section]
\newtheorem{definition}{Definition}[section]
\newtheorem{remark}{Remark}[section]
\newtheorem{lemma}{Lemma}[section]

\newtheorem{proposition}{Proposition}[section]

\newtheorem{corollary}{Corollary}[section]

\numberwithin{equation}{section}

\newcommand{\vs}{\vspace}

\begin{document}

\title{Global well-posedness, scattering and blow-up for
the energy-critical, Schr\"odinger
equation with general nonlinearity in the radial case\footnote{   This work was partially supported by NNSFC (No. 12171493).}}

\author{ Jun Wang$^{a}$, Zhaoyang Yin$^{a, b}$\footnote {Corresponding author. wangj937@mail2.sysu.edu.cn (J. Wang), mcsyzy@mail.sysu.edu.cn (Z. Yin)
} \\
{\small $^{a}$Department of Mathematics, Sun Yat-sen University, Guangzhou, 510275, China } \\
{\small $^{b}$School of Science, Shenzhen Campus of Sun Yat-sen University, Shenzhen, 518107, China } \\
}

	\date{}

	\maketitle

\date{}

 \maketitle \vs{-.7cm}

  \begin{abstract}
In this paper, we
study the well-posedness theory and the scattering asymptotics for the energy-critical, Schr\"odinger equation with general nonlinearity
\begin{equation*}
  \left\{\begin{array}{l}
i \partial_t u+\Delta u + f(u)=0,\ (x, t) \in \mathbb{R}^N \times \mathbb{R}, \\
\left.u\right|_{t=0}=u_0 \in  H ^1(\mathbb{R}^N),
\end{array}\right.
\end{equation*}
where $f:\mathbb{C}\rightarrow \mathbb{C}$ satisfies Sobolev critical growth condition. Using contraction mapping method and concentration compactness argument, we obtain
the well-posedness theory in proper function spaces and scattering asymptotics. This paper generalizes the conclusions in \cite{KCEMF2006}(Invent. Math).
\end{abstract}

{\footnotesize {\bf   Keywords:} Schr\"odinger equation, General nonlinearity, Well-posedness, Scattering

}

\section{ Introduction and main results}

This paper studies the well-posedness theory and the scattering asymptotics for the energy-critical, focusing, Schr\"odinger equation with general nonlinearity
\begin{equation}\label{eq1.1}
  \left\{\begin{array}{l}
i \partial_t u+\Delta u + f(u)=0,\ (x, t) \in \mathbb{R}^N \times \mathbb{R}, \\
\left.u\right|_{t=0}=u_0 \in  H ^1(\mathbb{R}^N),
\end{array}\right.
\end{equation}
where $f:\mathbb{C}\rightarrow \mathbb{C}$ satisfies Sobolev critical growth condition.

In \cite{KCEMF2006}, Kenig and Merle studied the $\dot{H}^1$ critical non-linear Schr\"odinger equation
$$
\left\{\begin{array}{l}
i \partial_t u+\Delta u \pm|u|^{\frac{4}{N-2}} u=0, \quad(x, t) \in \mathbb{R}^N \times \mathbb{R}, \\
\left.u\right|_{t=0}=u_0 \in \dot{H}^1(\mathbb{R}^N) .
\end{array}\right.
$$
Here the $-$ sign corresponds to the defocusing problem, while the $+$ sign corresponds to the focusing problem. They obtained  the global well-posedness, scattering and blow-up results in the radial case and $3\leq N\leq5$. Recently, Oh and Wang in \cite{TOYW2020} considered global well-posedness for one-dimensional cubic nonlinear Schr\"odinger equation by introducing a new function space. For $N=2$, \cite{JBAB2014} established global well-posedness results in the defocusing case, posed on the two-dimensional unit ball. For high dimensions case, Tao et.al in \cite{TTMV2007} established global well-posedness and scattering for solutions to the defocusing mass-critical nonlinear Schr\"odinger equation. For more well-posedness results, please refer to \cite{{XYHY2024},{DHNI2024},{HBGP2022},{JB1999},{JBM1999},{TCS2003},{TC1990}}.

Unlike other articles on well-posedness and scattering, this paper presents for the first time results on general critical nonlinear terms. This idea mainly comes from the variational method. In recent decades, many researchers have been interested in semi-linear Schr\"odinger problems with general nonlinear terms. From a mathematical perspective, it is very meaningful. In order to obtain the solution of the equation, it is necessary to impose appropriate conditions on the nonlinear term. The most classic condition is the Ambrosetti-Rabinowitz(AR) condition:
\begin{equation*}
\text {There exists } \mu>2 \text { such that } G(t)  \geq \mu g(t)  >0 \text { for all } t>0 \text {, }
\end{equation*}
where $G(t)=\int_{0}^tg(s)ds$. As is well known, (AR) condition plays a crucial role in getting a bounded (PS)-sequence, so we can use the Nehari manifold method to obtain the solution of the equation. As we can see, the (AR) condition ensures that nonlinearity has good properties. However, we use another weaker Berestycki-Lions(BL) condition than (AR) in this paper, namely
\begin{equation*}
\text {There exists } \xi>0 \text { such that } G(\xi)  >0 .
\end{equation*}
As mentioned in \cite{HB1983}, under the (BL) condition, we can use Pohozaev manifold to obtain the existence of standing wave solutions(even the existence of infinite solutions), which is crucial for obtaining variational estimates in Section 6.

Now, we give the most general assumption of growth, i.e., assume that $f: \ \mathbb{C}\rightarrow \mathbb{C}$ is $C^1$ and satisfies:
\begin{itemize}
\item[$(F_1)$]\ $\lim\limits_{|z| \rightarrow 0^{+}} \frac{|f(z)\overline{z}|}{|z|^2 }=0$.

\item[$(F_2)$]\ $\limsup\limits_{|z| \rightarrow \infty} \frac{|f(z)\overline{z}|}{|z|^{2^*}}<\infty$, where $2^*=\frac{2N}{N-2}$.

\item[$(F_3)$]\ $\lim\limits_{|\nabla z| \rightarrow 0^{+}} \frac{|\nabla (f(z))|}{|\nabla z| }=0$.

\item[$(F_4)$]\ $\limsup\limits_{|\nabla z| \rightarrow \infty} \frac{|\nabla (f(z))|}{|z|^{2^*-2}||\nabla z|}<\infty$.
\end{itemize}
\begin{remark}\label{r1.1}
By $(F_1)$, $(F_2)$, $(F_3)$, $(F_4)$, there exist $\varepsilon,\ C>0$ such that
\begin{equation*}
\frac{|f(z)\overline{z}|}{|z|^2 }<\varepsilon,\ \frac{|\nabla (f(z))|}{|\nabla z| }<\varepsilon
\end{equation*}
and
\begin{equation*}
\frac{|f(z)\overline{z}|}{|z|^{2^*}}<C,\ \frac{|\nabla (f(z))|}{|z|^{2^*-2}||\nabla z|}<C.
\end{equation*}
Hence, it follows from $f\in C^1$ that there exist $C_1,\ C_2,\ C_3,\ C_4>0$ such that
\begin{equation*}
  |f(z)|\leq C_1|z|+C_2|z|^{2^*-1},\  |\nabla (f(z))|\leq C_3|\nabla z|+C_4|z|^{2^*-2}|\nabla z|.
\end{equation*}
\end{remark}
\begin{remark}\label{r1.2}
Let $f(u)=|u|^{\frac{4}{N-2}} u+|u|^{\frac{2}{N-2}} u$, then clearly
\begin{equation*}
  |f(u)| \leq|u|^{\frac{N+2}{N-2}}+|u|^{\frac{N}{N-2}}, \left|\partial_z f(u)\right| \leq C|u|^{\frac{4}{N-2}}+|u|^{\frac{2}{N-2}},\left|\partial_{\bar{z}} f(u)\right| \leq C|u|^{\frac{4}{N-2}}+|u|^{\frac{2}{N-2}}.
\end{equation*}
Moreover, for $3 \leq N \leq 4$,
$$
\left.\begin{array}{l}
\left|\partial_z f(u)-\partial_z f(v)\right| \\
\left|\partial_{\bar{z}} f(u)-\partial_{\bar{z}} f(v)\right|
\end{array}\right\} \leq C|u-v| \cdot\left\{|u|^{\frac{6-N}{N-2}}+|v|^{\frac{6-N}{N-2}}+|u|^{\frac{4-N}{N-2}}+|v|^{\frac{4-N}{N-2}}\right\} .
$$
Also,
\begin{equation*}
  |f(u)-f(v)| \leq|u-v|\left\{|u|^{\frac{4}{N-2}}+|v|^{\frac{4}{N-2}}+|u|^{\frac{2}{N-2}}+|v|^{\frac{2}{N-2}}\right\}.
\end{equation*}
In addition,
\begin{eqnarray*}
&&\nabla_x(f(u(x)))-\nabla_x(f(v(x)))\\
&= & (\nabla f)(u(x)) \nabla u-(\nabla f)(v(x)) \nabla v \\
&= & (\nabla f)(u(x)) \nabla u-(\nabla f)(u(x)) \nabla v  +\{\nabla f(u(x)))-\nabla f(v(x))\} \nabla v,
\end{eqnarray*}
so
\begin{eqnarray*}
 \left|\nabla_x f(u(x))-\nabla_x f(v(x))\right| &\leq& C|u|^{\frac{4}{N-2}}|\nabla u-\nabla v|+C|\nabla v|\left\{|u|^{\frac{6-N}{N-2}}+ |v|^{\frac{6-N}{N-2}}\right\}|u-v|\\
 & &+ C|u|^{\frac{2}{N-2}}|\nabla u-\nabla v|+C|\nabla v|\left\{|u|^{\frac{6-N}{N-2}}+ |v|^{\frac{4-N}{N-2}}\right\}|u-v|.
\end{eqnarray*}
\end{remark}

In the subsequent proof, energy conservation is necessary. If the nonlinear term is of power type, there is obviously energy conservation. This paper considers more general nonlinear terms. Inspired by \cite{JGGV1979}, we need to make the following new assumption:
\begin{itemize}
\item[$(F_5)$]\ For all $z \in \mathbb{C}, f(\bar{z})=\overline{f(z)}$. For all $z \in \mathbb{C}$ and all $\omega \in \mathbb{C}$ with $|\omega|=1$, $f(\omega z)=\omega f(z)$.
\end{itemize}
It follows from $(F_5)$ that $f(z)$ is of the form $f(z)=z \tilde{f}(|z|)$ where $\tilde{f}$ is a real function defined on $(0, \infty)$. Moreover, note that $f$ is continuous, one can define a function $F$ from $\mathbb{C}$ to $\mathbb{R}$ by
\begin{eqnarray*}
F(\rho):=2 \int_0^\rho \sigma \tilde{f}(\sigma) d \sigma && \text { for all } \rho \geq 0, \\
F(z):=F(|z|) && \text { for all } z \in \mathbb{C} .
\end{eqnarray*}
The function $f$ can be recovered from $F$ by the relation
$$
f(z)=\frac{\partial F(z)}{\partial \bar{z}} .
$$
Hence, by using  Propositions 3.3 and 3.4 in \cite{JGGV1979}, we can obtain the conservation of the $L^2$-norm and of the energy in differential form for the regularized equation. That is, if the solution $u$ of \eqref{eq1.1} has sufficient decay at infinity and smoothness, it satisfies the conservation of mass
\begin{equation}\label{eq1.2}
  M(u(t)):=\|u(t)\|_{L^2}=\|u_0\|_{L^2}
\end{equation}
and the conservation of energy
\begin{equation}\label{eq1.3}
  E(u(t))=E(u_0),
\end{equation}
where $E(u(t))$ is defined by
\begin{equation*}
  E(u(t)):=  \frac{1}{2} \int_{\mathbb{R}^N}\left|\nabla u(t, x)\right|^2 d x-   \int_{\mathbb{R}^N}F(u(t))  d x,\ F(u)=F(|u|)=2\int_{0}^{|u|}s\tilde{f}(s)ds
\end{equation*}
and the energy space is $H^1$.

The main results of this paper are as follows.
\begin{theorem}\label{t1.1}
 Assume $u_0 \in H_{rad}^1(\mathbb{R}^N)$,\ $N\geq3$ and $(F_1)-(F_5)$ hold, then there exists a unique solution $u$ to \eqref{eq1.1} satisfying
 \begin{equation*}
    u \in C([0,T] , H_{rad}^1(\mathbb{R}^N))\cap L_{loc}^{\frac{2(N+2)}{N-2}}([0,T],W^{1,{\frac{2N(N+2)}{N^2+4}}}).
 \end{equation*}
Moreover, there exists the globally solution $u$ to \eqref{eq1.1}  if  $\|u_0\|_{L^2}$ is small enough.
\end{theorem}
\begin{remark}\label{r1.3}
The conclusion of Theorem \ref{t1.1} also holds for defocusing cases, please refer to the proof of Lemma \ref{L3.2} for details.
\end{remark}
\begin{theorem}\label{t1.2}
Assume that $N \geq 3$ and $u_0 \in H_{\text {rad }}^1$ and $(F_1)-(F_5)$ hold. Let $u$ be such that the corresponding solution to \eqref{eq1.1} exists on the maximal time $T^*$. If $E(u_0)<0$, then one of the following statements holds true:

$(1)$ $u(t)$ blows up in finite time in the sense that $T^*<+\infty$ must hold.

$(2)$ $u(t)$ blows up infinite time such that
\begin{equation}\label{eq4.1}
  \sup _{t \geq 0}\left\|\nabla u(\cdot, t)\right\|_{L^2}=\infty .
\end{equation}
\end{theorem}
\begin{remark}\label{r1.4}
Under the assumption of Theorem \ref{t1.1}, we can define a maximal interval $I(u_0)=(t_0-T_{-}(u_0),t_0+T_{+}(u_0))$, with $T_{\pm}(u_0)>0$, where the solution is defined. If $T_{+}(u_0)<\infty$, then by standard finite blow-up criterion, we know that
\begin{equation*}
  \|u\|_{L_{[t_0,t_0+T_{+}(u_0)]}^\frac{2(N+2)}{N-2}W^{1,{\frac{2N(N+2)}{N^2+4}}}}=+\infty,
\end{equation*}
the corresponding result holds for $T_{-}(u_0)$.
\end{remark}

\begin{theorem}\label{t1.3}
Let $(F_1)-(F_6)$ hold, $N=3,4$. Assume that
\begin{equation*}
  E\left(u_0\right)<E(W), \int_{\mathbb{R}^N}\left|\nabla u_0\right|^2dx<\int_{\mathbb{R}^N}|\nabla W|^2dx
\end{equation*}
and $u_0$ is radial. Then there exist $u_{0,+}, u_{0,-}$ in $H^1$ such that
$$
\lim\limits_{t \rightarrow+\infty}\left\|u(t)-e^{i t \Delta} u_{0,+}\right\|_{H^1}=0, \  \lim\limits_{t \rightarrow-\infty}\left\|u(t)-e^{i t \Delta} u_{0,-}\right\|_{H^1}=0 .
$$
\end{theorem}
\begin{remark}\label{r1.5}
Theorem \ref{t1.3} is a part of Corollary \ref{c9.1}. In fact, according to the analysis in Remark \ref{r4.1}, we only need to prove that the conclusion holds when t tends towards positive infinity, and the other half of the conclusion can be obtained similarly.
\end{remark}

In section 2, we provide some notations and some important lemma in the proof of main theorems. In the next section, we aim to prove theorem \ref{t1.1}, that is well-posedness. In section 4, we get the blow up solutions in infinite time. Finally, we will obtain the scattering asymptotics. To achieve this goal, we need some new variational estimates and compactness results, that is, sections 6, 7 and 8.

\section{Preliminary}
In this section, we provide some notations and some important lemma in the proof of main theorems.

\begin{lemma}\label{L2.1}(Strichartz estimate \cite{{MKTT1998},{TCS2003}}). We say that a pair of exponents $(q, r)$ is admissible if $\frac{2}{q}+\frac{N}{r}=\frac{N}{2}$ and $2 \leq q, r \leq \infty$. Then, if $2 \leq r \leq \frac{2 N}{N-2}$ $(N \geq 3)$$($or $2 \leq r<\infty, N=2$ and $2 \leq r \leq \infty, N=1)$ we have\\
i)
$$
\left\|e^{i t \Delta} h\right\|_{L_t^q L_x^r} \leq C\|h\|_{L^2},
$$
ii)
$$
\left\|\int_{-\infty}^{+\infty} e^{i(t-\tau) \Delta} g(-, \tau) d \tau\right\|_{L_t^q L_x^r}+\left\|\int_0^t e^{i(t-\tau) \Delta} g(-, \tau) d \tau\right\|_{L_t^q L_x^r} \leq C\|g\|_{L_t^{q^{\prime}} L_x^{r^{\prime}}},
$$
iii)
$$
\left\|\int_{-\infty}^{+\infty} e^{i t \Delta} g(-, \tau) d \tau\right\|_{L_x^2} \leq C\|g\|_{L_t^{q^{\prime}} L_x^{\prime^{\prime}}} .
$$
\end{lemma}
\begin{remark}\label{R2.1}
In the estimate ii) in Lemma \ref{L2.1}, one can actually show: (\cite{MKTT1998}) ii')
$$
\left\|\int_{-\infty}^{+\infty} e^{i(t-\tau) \Delta} g(-, \tau) d \tau\right\|_{L_t^q L_x^r} \leq C\|g\|_{L_t^{m^{\prime}} L_x^{n^{\prime}}},
$$
where $(q, r),(m, n)$ are any pair of admissible indices as in i) of Lemma 2.1.
\end{remark}

\begin{lemma}\label{L2.2}(Sobolev embedding). For $v \in C_0^{\infty}\left(\mathbb{R}^{N+1}\right)$, we have
\begin{equation*}
  \|v\|_{L_t^{\frac{2(N+2)}{N-2}}L_x^{\frac{2(N+2)}{N-2}}}\leq C\|\nabla_xv\|_{L_t^{\frac{2(N+2)}{N-2}}L_x^{\frac{2N(N+2)}{N^2+4}}}(N\geq 3).
\end{equation*}
(Note that $\frac{2(N+2)}{N-2}=q, \frac{2 N(N+2)}{N^2+4}=r$ is admissible.)
\end{lemma}

\begin{lemma}\label{L2.3}(\cite{VDD2018}) Let $F(z)=|z|^k z$ with $k>0, s \geq 0$ and $1<p, p_1<\infty, 1<q_1 \leq \infty$ satisfying $\frac{1}{p}=\frac{1}{p_1}+\frac{k}{q_1}$. If $k$ is an even integer or $k$ is not an even integer with $[s] \leq k$, then there exists $C>0$ such that for all $u \in \mathscr{S}$,
$$
\|F(u)\|_{\dot{W}^{s, p}} \leq C\|u\|_{L^{q_1}}^k\|u\|_{\dot{W}^{s, p_1}} .
$$
A similar estimate holds with $\dot{W}^{s, p}, \dot{W}^{s, p_1}$-norms replaced by $W^{s, p}, W^{s, p_1}$ norms.
\end{lemma}

\begin{lemma}\label{L2.4}(\cite{TBDH2016}) Let $N \geq 1$ and $f: \mathbb{R}^N \rightarrow \mathbb{R}$ satisfy $\nabla f \in W^{1, \infty}\left(\mathbb{R}^N\right)$. Then, for all $u \in H^{\frac{1}{2}}\left(\mathbb{R}^N\right)$, it holds that
$$
\left|\int_{\mathbb{R}^N} \bar{u}(x) \nabla f(x) \cdot \nabla u(x) d x\right| \leqslant C\left(\left\||\nabla|^{\frac{1}{2}} u\right\|_{L^2}^2+\|u\|_{L^2}\left\||\nabla|^{\frac{1}{2}} u\right\|_{L^2}\right),
$$
with some constant $C>0$ depending only on $\|\nabla f\|_{W^{1, \infty}}$ and $N$.
\end{lemma}

\begin{definition}\label{D2.1}
Let $v_0 \in H^1, v(t)=e^{i t \Delta} v_0$ and let $\left\{t_n\right\}$ be a sequence, with $\lim\limits_{n \rightarrow \infty} t_n=\bar{t} \in[-\infty,+\infty]$. We say that $u(x, t)$ is a non-linear profile associated with $\left(v_0,\left\{t_n\right\}\right)$ if there exists an interval $I$, with $\bar{t} \in I$ (if $\bar{t}= \pm \infty$, $I=[a,+\infty)$ or $(-\infty, a])$ such that $u$ is a solution of \eqref{eq1.1} in $I$ and
$$
\lim\limits_{n \rightarrow \infty}\left\|u\left(-, t_n\right)-v\left(-, t_n\right)\right\|_{H^1}=0 .
$$
\end{definition}
Next, we give some definitions that play a crucial role in the proof of concentration compactness in section 6.

\begin{definition}\label{D2.2}
(i) We call scale, every sequence $\mathbf{h}=\left(h_n\right)_{n \geqslant o}$ of positive numbers and core, every sequence $\mathbf{z}=\left(z_n\right)_{n \geqslant 0}=\left(t_n, x_n\right)_{n \geqslant 0} \subset \mathbb{R} \times \mathbb{R}^N$.

(ii) We say that two pairs $(\mathbf{h}, \mathbf{z})$ and $\left(\mathbf{h}^{\prime}, \mathbf{z}^{\prime}\right)$ are orthogonal if
$$
\frac{h_n}{h_n^{\prime}}+\frac{h_n^{\prime}}{h_n}+\left|\frac{t_n-t_n^{\prime}}{\left(h_n\right)^2}\right|+\left|\frac{x_n-x_n^{\prime}}{h_n}\right| \rightarrow+\infty,\ n \rightarrow \infty .
$$
\end{definition}

\begin{definition}\label{D2.3}
(i) A pair $(q, r)$ is $L^2$-admissible, if $r \in[2,\frac{2N}{N-2})$ and $q$ satisfies
$$
\frac{2}{q}+\frac{N}{r}=\frac{N}{2} .
$$
(ii) A pair $(q, r)$ is $H^1$-admissible, if $r \in[\frac{2N}{N-2},+\infty)$ and $q$ satisfies
$$
\frac{2}{q}+\frac{N}{r}=\frac{N-2}{2} .
$$
\end{definition}
The main result proved in \cite{{JGGV1984},{KY1987}} is the following:

\begin{proposition}\label{P2.1}
Let $(q, r)$ be an $L^2$-admissible pair. There exists $C=$ $C(r)$, such that
 \begin{equation}\label{eq2.1}
   \left\|e^{i t \Delta} h\right\|_{L_t^q L_x^r} \leq  C\| \varphi \|_{L^2\left(\mathbb{R}^N\right)}
 \end{equation}
for every $\varphi \in L^2(\mathbb{R}^N)$.
\end{proposition}
A direct consequence of Proposition \ref{P2.1}, via Sobolev's inequality, is the following.

\begin{proposition}\label{P2.2}
Let $(q, r)$ be an $H^1$-admissible pair. There exists $C=$ $C(r)$, such that
 \begin{equation}\label{eq2.2}
   \left\|e^{i t \Delta} h\right\|_{L_t^q L_x^r} \leq   C\|\nabla \varphi\|_{L^2\left(\mathbb{R}^N\right)}
 \end{equation}
for every $\varphi \in \dot{H}^1(\mathbb{R}^N)$.
\end{proposition}

\begin{definition}\label{D2.4}
Let $(\mathbf{h}, \mathbf{z})$ be a pair of scales-cores, such that the quantity $\frac{t_n }{h_n^2}$ has a limit in $[-\infty,+\infty]$, when $n$ goes to the infinity. Let $V$ be a solution of linear Schr\"odinger equation \eqref{eq7.5}. We say that $U$ is the nonlinear profile associated to $\mathscr{W}=(V, \mathbf{h}, \mathbf{z})$ if $U$ is the solution of the nonlinear Schr\"odinger equation \eqref{eq1.1} satisfying
$$
\left\|(U-V)\left(-\frac{t_n}{h_n^2}\right)\right\|_{H^1(\mathbb{R}^N)} \rightarrow  0,\ n \rightarrow \infty .
$$
\end{definition}

\textbf{Notations:}

$\bullet$ Throughout this paper, we use $C$ to denote the universal constant and $C$ may change line by line.

$\bullet$ We also use notation $C(B)$(or $C_B$) to denote a constant depends on $B$.

$\bullet$ We use usual $L^p$ spaces and Sobolev spaces $H^1$. $p^{\prime}$ for the dual index of $p \in(1,+\infty)$ in the sense that $\frac{1}{p^{\prime}}+\frac{1}{p}=1$.

$\bullet$ We use notation $A\Subset B$ to denote $A$ is a open subset of $B$.

\section{Well-posedness}

First, we get the existence and uniqueness of solution to \eqref{eq1.1} by using contraction mapping method.
\begin{lemma}\label{L3.1}   Assume $u_0 \in H_{rad}^1(\mathbb{R}^N)$,\ $N\geq3$, then there exists a unique solution $u$ to \eqref{eq1.1} satisfying 
 \begin{equation*}
    u \in C([0,T] , H_{rad}^1(\mathbb{R}^N))\cap L_{loc}^{\frac{2(N+2)}{N-2}}([0,T],W^{1,{\frac{2N(N+2)}{N^2+4}}}).
 \end{equation*}
\end{lemma}
\begin{proof}
\eqref{eq1.1} is equivalent to the integral equation
$$
u(t)=e^{i t \Delta} u_0+i\int_0^t e^{i\left(t-t^{\prime}\right) \Delta} f(u) d t^{\prime}.
$$
Now, we consider  the complete metric space
\begin{equation*}
  B(T, \rho)=\left\{C([0,T] , H^1(\mathbb{R}^N))\cap L_T^{\frac{2(N+2)}{N-2}}W^{1,{\frac{2N(N+2)}{N^2+4}}} :\|u\|_{L_T^\infty H^1}+\|u\|_{L_T^\frac{2(N+2)}{N-2}W^{1,{\frac{2N(N+2)}{N^2+4}}}}\leq\rho\right\}
\end{equation*}
equipped with the distance
$$
d_B(u, v):=\|u-v\|_{L_T^{\infty} L^2}+\|u-v\|_{L_T^\frac{2(N+2)}{N-2}L^{\frac{2N(N+2)}{N^2+4}}}
$$
and
\begin{equation*}
  \Phi_{u_0}(v)=e^{i t \Delta} u_0+i\int_0^t e^{i\left(t-t^{\prime}\right) \Delta} f(v) d t^{\prime}.
\end{equation*}
Next, we prove that $\Phi_{u_0}(v):  B(T, \rho) \rightarrow  B(T, \rho)$ and is a contraction. In fact, by $(F_1)$ and $(F_2)$, we obtain
\begin{equation*}
|f(v)|\leq C_1|v| +C_2|v|^{2^*-2}|v|,\  |\nabla (f(z))|\leq C_1|\nabla z|+C_2|z|^{2^*-2}|\nabla z|.
\end{equation*}
Hence, it follows from $i)$ in lemma \ref{L2.1}, Remark \ref{R2.1}, lemma \ref{L2.2} and H\"older inequality that
\begin{eqnarray}\label{eq3.1}
&&\left\| \Phi_{u_0}(v)\right\|_{L_T^\infty H^1\cap L_T^\frac{2(N+2)}{N-2}W^{1,{\frac{2N(N+2)}{N^2+4}}}}\nonumber\\
&=&\left\|e^{i t \Delta} u_0+i\int_0^t e^{i\left(t-t^{\prime}\right) \Delta} f(v) d t^{\prime}\right\|_{L_T^\infty H^1\cap L_T^\frac{2(N+2)}{N-2}W^{1,{\frac{2N(N+2)}{N^2+4}}}}\nonumber\\
&\leq&C\|u_0\|_{H^1}+\left\|\int_0^t e^{i\left(t-t^{\prime}\right) \Delta} f(v) d t^{\prime}\right\|_{L_T^\infty H^1\cap L_T^\frac{2(N+2)}{N-2}W^{1,{\frac{2N(N+2)}{N^2+4}}}} \nonumber\\
&\leq&C\|u_0\|_{H^1}+\left\|\int_0^t e^{i\left(t-t^{\prime}\right) \Delta} v d t^{\prime}\right\|_{L_T^\infty H^1}+\left\|\int_0^t e^{i\left(t-t^{\prime}\right) \Delta} v\cdot |v|^{\frac{4}{N-2}} d t^{\prime}\right\|_{L_T^\frac{2(N+2)}{N-2}W^{1,{\frac{2N(N+2)}{N^2+4}}}} \nonumber\\
&\leq&C\|u_0\|_{H^1}+\left\|v\right\|_{L_T^1H^1}+C\left\| v\cdot |v|^{\frac{4}{N-2}}  \right\|_{L_T^2 W^{1,\frac{2 N}{N+2}}}   \nonumber\\
&\leq&C\|u_0\|_{H^1}+T\left\|v  \right\|_{L_T^\infty H^1}+C\left\|\| v\|_{W^{1,\frac{2N(N+2)}{N^2+4}}}\cdot \|v\|_{L^{\frac{2(N+2)}{N-2}}}^{\frac{4}{N-2}}  \right\|_{L_I^2 }  \nonumber\\
&\leq&C\|u_0\|_{H^1}+T\left\|v  \right\|_{L_T^\infty H^1}+C\left\|v\right\|_{L_T^{\frac{2(N+2)}{N-2}}W^{1,{\frac{2N(N+2)}{N^2+4}}}}\cdot\left\| v  \right\|_{L_T^{\frac{2(N+2)}{N-2}}L^{\frac{2(N+2)}{N-2}}}^{\frac{4}{N-2}}\nonumber\\
&\leq&C\|u_0\|_{H^1}+T\left\|v  \right\|_{L_T^\infty H^1}+C\left\|v\right\|_{L_T^{\frac{2(N+2)}{N-2}}W^{1,{\frac{2N(N+2)}{N^2+4}}}}^{\frac{N+2}{N-2}}\\
&\leq&C(\|u_0\|_{H^1}+T\rho+\rho^{\frac{N+2}{N-2}})\nonumber  .
\end{eqnarray}
Choosing $T$ and $\rho$ such that $C(\|u_0\|_{H^1}+T\rho+\rho^{\frac{N+2}{N-2}})\leq\rho$, then $\Phi_{u_0}(v)$ maps $B(T, \rho)$ to $B(T, \rho)$.

Now we prove that $\Phi_{u_0}(v)$ is a contraction map for sufficiently small $T$. Let $u, v \in B(T, \rho)$, by Remark \ref{R2.1} and lemma \ref{L2.3}, then we have
 \begin{eqnarray*}
 &&d_B(\Phi_{u_0}(u), \Phi_{u_0}(v))\nonumber\\
 &=&\|\Phi_{u_0}(u)-\Phi_{u_0}(v)\|_{L_T^{\infty} L^2}+\|\Phi_{u_0}(u)-\Phi_{u_0}(v)\|_{L_T^\frac{2(N+2)}{N-2}L^{\frac{2N(N+2)}{N^2+4}}}\nonumber\\
 &\leq&\left\|\int_0^t e^{i\left(t-t^{\prime}\right) \Delta}|u-v|  d t^{\prime}\right\|_{L_T^{\infty} L^2}+\left\|\int_0^t e^{i\left(t-t^{\prime}\right) \Delta}|u\cdot |u|^{\frac{4}{N-2}}-v\cdot |v|^{\frac{4}{N-2}}| d t^{\prime}\right\|_{L_T^{\infty} L^2} \nonumber\\
 &&+\left\|\int_0^te^{i\left(t-t^{\prime}\right) \Delta}(|u-v|+|u\cdot |u|^{\frac{4}{N-2}}-v\cdot |v|^{\frac{4}{N-2}}|) dt^{\prime}\right\|_{L_T^\frac{2(N+2)}{N-2}L^{\frac{2N(N+2)}{N^2+4}}}\nonumber\\
 &\leq&C\|u-v\|_{L_T^{1} L^2}+\|u\cdot |u|^{\frac{4}{N-2}}-v\cdot |v|^{\frac{4}{N-2}}\|_{L_T^{2} L^{\frac{2N}{N+2}}}+\|u\cdot |u|^{\frac{4}{N-2}}-v\cdot |v|^{\frac{4}{N-2}}\|_{L_T^{2} L^{\frac{2N}{N+2}}}\nonumber\\
 &&+\|u-v\|_{L_T^{2} L^{\frac{2N}{N+2}}} \nonumber\\
 &\leq&CT\|u-v\|_{L_T^{\infty} L^2}+C\|u\cdot |u|^{\frac{4}{N-2}}-v\cdot |v|^{\frac{4}{N-2}}\|_{L_T^{2} L^{\frac{2N}{N+2}}}+CT\|u-v\|_{L_T^{\infty} H^1}\nonumber\\
 &\leq&CT\|u-v\|_{L_T^{\infty} H^1}+C\| u-v\|_{L_T^\frac{2(N+2)}{N-2}L^{\frac{2N(N+2)}{N^2+4}}} \left(\left\|u\right\|_{L_T^{\frac{2(N+2)}{N-2}}L^{\frac{2N(N+2)}{N^2+4}}}^{\frac{4}{N-2}}+\left\|v\right\|_{L_T^{\frac{2(N+2)}{N-2}}L^{\frac{2N(N+2)}{N^2+4}}}^{\frac{4}{N-2}}\right)\\ &\leq&C\| u-v\|_{L_T^\frac{2(N+2)}{N-2}L^{\frac{2N(N+2)}{N^2+4}}} \left(\left\|u\right\|_{L_T^{\frac{2(N+2)}{N-2}}W^{1,{\frac{2N(N+2)}{N^2+4}}}}^{\frac{4}{N-2}}+\left\|v\right\|_{L_T^{\frac{2(N+2)}{N-2}}W^{1,{\frac{2N(N+2)}{N^2+4}}}}^{\frac{4}{N-2}}\right)\\
 &&CT\|u-v\|_{L_T^{\infty} H^1}\\
 &\leq&C(T+\rho^{\frac{4}{N-2}})d_B(u,v).
 \end{eqnarray*}
 Choosing $T$ and $\rho$ such that $C(T+\rho^{\frac{4}{N-2}})\leq\frac{1}{2}$, the above estimate implies that $\Phi_{u_0}(v)$ is a contraction. Therefore, $\Phi_{u_0}(v)$ has a fixed point in $B(T,\rho)$.
\end{proof}

The above lemma implies local well-posedness for \eqref{eq1.1}. Next, we will show the global well-posedness.

\begin{lemma}\label{L3.2} Let $N \geq 3$. Then for any $u_0 \in H_{\text {rad }}^1$, there exists the globally solution $u$ to \eqref{eq1.1}  if  $\|u_0\|_{L^2}$ is small enough.
\end{lemma}
\begin{proof}
Assume that $T^*$ be the maximal existence time. We will show that $T^*=\infty$ by contradiction. Let $T^*<\infty$, then according to the local wellposedness, it holds
\begin{equation}\label{eq3.2}
  \|u\|_{L_{T^*}^{\frac{2(N+2)}{N-2}} W^{1,{\frac{2N(N+2)}{N^2+4}}}}=\infty .
\end{equation}
\textbf{Case 1.} Consider  defocusing case. Using the conservation laws \eqref{eq1.2} and \eqref{eq1.3}, for any $t<T^*$, the solution $u$ satisfies
$$
\frac{1}{2}\|u(t)\|_{H^1}^2 \leq \frac{1}{2}\|u(t)\|_{L^2}^2+E(u)=\frac{1}{2}\|u_0\|_{L^2}^2+E(u_0) .
$$
From the estimate \eqref{eq3.1}, we have
 \begin{eqnarray*}
&& \|u\|_{L_{T^*}^{\frac{2(N+2)}{N-2}} W^{1,{\frac{2N(N+2)}{N^2+4}}}} \\
 & \leq&C\|u_0\|_{H^1}+T\left\|v  \right\|_{L_T^\infty H^1}+C\left\|u\right\|_{L_T^{\frac{2(N+2)}{N-2}}W^{1,{\frac{2N(N+2)}{N^2+4}}}}^{\frac{N+2}{N-2}} \\
&\leq&C\left((\|u_0\|_{L^2}^2+E(u_0))^{\frac{1}{2}}+T\left(\|u_0\|_{L^2}^2+E(u_0)\right)^{\frac{1}{2}} +\left\|u\right\|_{L_T^{\frac{2(N+2)}{N-2}}W^{1,{\frac{2N(N+2)}{N^2+4}}}}^{\frac{N+2}{N-2}}\right).
 \end{eqnarray*}
Note that $\left\|u\right\|_{L_T^{\frac{2(N+2)}{N-2}}W^{1,{\frac{2N(N+2)}{N^2+4}}}}^{\frac{N+2}{N-2}}\leq\rho$,
so for sufficiently small $T$ depending only on $\|u_0\|_{L^2}^2+E(u_0)$, we have
$$
\|u\|_{L_{(T_{j-1}, T_j)}^{\frac{2(N+2)}{N-2}}W^{1,{\frac{2N(N+2)}{N^2+4}}}} \leq C\left(\|u_0\|_{L^2}^2+E(u_0)\right)^{\frac{1}{2}},
$$
where $T_j-T_{j-1}=T$. If $T^*<\infty$, then there exists $M \in \mathbb{N}$ such that $(M-1) T<T^* \leq M T$. Let $T_j=(j-1) T, j=1, \ldots, M-1$ and $T_M=T^*$. Then we have
$$
\begin{aligned}
\|u\|_{L_{(0, T^*)}^{\frac{2(N+2)}{N-2}}W^{1,{\frac{2N(N+2)}{N^2+4}}}}^{\frac{2(N+2)}{N-2}} & \leq \sum_{1 \leq j \leq M}\|u\|_{L_{(T_{j-1}, T_j)}^{\frac{2(N+2)}{N-2}}W^{1,{\frac{2N(N+2)}{N^2+4}}}}^{\frac{2(N+2)}{N-2}} \\
& \leq\left(M C\left(\|u_0\|_{L^2}^2+E(u_0)\right)^{\frac{1}{2}}\right)^p<\infty .
\end{aligned}
$$
This contradicts to \eqref{eq3.2}.

\textbf{Case 2.} Consider focusing case. By $(F_1)-(F_5)$,
\begin{eqnarray*}
E(u_0)=E(u(t))&=&  \frac{1}{2} \int_{\mathbb{R}^N}\left|\nabla u(t, x)\right|^2 d x-   \int_{\mathbb{R}^N}F(u(t))  d x\\
&\geq&  \frac{1}{2} \|\nabla u\|_{L^2}^2-\varepsilon\|\nabla u\|_{L^2}^2-C_{\varepsilon}\|\nabla u\|_{L^2}^{2^*},
\end{eqnarray*}
which implies $\|u\|_{H^1}\leq CE(u_0)$, so we get
$$
\|u(t)\|_{H^1}^2 \leq  \|u(t)\|_{L^2}^2+CE(u)=\frac{1}{2}\|u_0\|_{L^2}^2+CE(u_0) .
$$
Similarly to the proof of the Case 1, which completes the proof of the lemma.
\end{proof}
\begin{proof}[\bf Proof of Theorem \ref{t1.1}]
Combining Lemmas \ref{L3.1} and \ref{L3.2}, Theorem \ref{t1.1} holds.
\end{proof}
\section{Blow up solutions}
In this section, we will investigate the blow up solutions of \eqref{eq1.1}.  Let $\psi \in C_0^{\infty}(\mathbb{R}^N)$ be radial and satisfy
$$
\psi(r)=\left\{\begin{array}{ll}
\frac{1}{2} r^2, & \text { for } r \leqslant 1 \\
0, & \text { for } r \geqslant 2
\end{array} \quad \text { and } \quad \psi^{\prime \prime}(r) \leqslant 1 \quad \text { for } r=|x| \geqslant 0 .\right.
$$
For a fixed $R>0$, we define the rescaled function $\psi_R: \mathbb{R}^N \rightarrow \mathbb{R}$ by setting
 \begin{equation}\label{eq4.2}
\psi_R(r):=R^2 \psi\left(\frac{r}{R}\right) .
 \end{equation}
Next we will show that
 \begin{equation}\label{eq4.3}
   1-\psi_R^{\prime \prime}(r) \geqslant 0, \quad 1-\frac{\psi_R^{\prime}(r)}{r} \geqslant 0, \quad N-\Delta \psi_R(r) \geq  0 \quad \text { for all } r \geqslant 0 .
 \end{equation}
Indeed, this first inequality follows from $\psi_R^{\prime \prime}(r)=\psi^{\prime \prime}(\frac{r}{ R}) \leq 1$. We obtain the second inequality by integrating the first inequality on $[0, r]$ and using that $\psi_R^{\prime}(0)=0$. Finally, we see that last inequality follows from
$$
N-\Delta \psi_R(r)=1-\psi_R^{\prime \prime}(r)+ (N-1)\left(1-\frac{1}{r} \psi_R^{\prime}(r)\right) \geqslant 0 .
$$
Besides \eqref{eq4.3}, $\psi_R$ admits the following properties, which can be easily checked.
We define
$$
\mathcal{M}_{\psi_R}[u(t)]:=2\mathrm{Im}  \int_{\mathbb{R}^N} \overline{u}(t) \nabla \psi_R \cdot \nabla u(t) d x=2\mathrm{Im} \int_{\mathbb{R}^N} \overline{u}(t) \partial_j \psi_R \partial_j u(t) d x .
$$
Define the self-adjoint differential operator
$$
\Gamma_{\psi_R}:=i\left(\nabla \cdot \nabla \psi_R+\nabla \psi_R \cdot \nabla\right),
$$
which acts on functions according to
$$
\Gamma_{\psi_R} f= i\left(\nabla \cdot\left(\left(\nabla \psi_R\right) f\right)+\left(\nabla \psi_R\right) \cdot(\nabla f)\right) .
$$
It's easy to check that
$$
\mathcal{M}_{\psi_R}[u(t)]=\left\langle u(t), \Gamma_{\psi_R} u(t)\right\rangle .
$$

Next, we show the following useful lemma.
\begin{lemma}\label{L4.1}
Let $N \geq 3$, and $u \in H_{\text {rad }}^1$ is a solution of \eqref{eq1.1}. Let $\psi_R$ be as in \eqref{eq4.2}, $T^*$ be the maximal existence time of solution $u(t)$ in $H_{\text {rad }}^1$. Then for sufficiently large $R$, it holds
$$
\frac{d}{d t} \mathcal{M}_{\psi_R}[u(t)] \leq 8 E(u(t)), \quad t \in\left[0, T^*\right) .
$$
\end{lemma}
\begin{proof}
By taking the derivative of $\mathcal{M}_{\psi_R}[u(t)]$ with respect to time $t$ and using the equation of $u(t)$, for any $t \in[0, T)$, it follows that
 \begin{eqnarray*}
\frac{d}{d t} \mathcal{M}_{\psi_R}[u(t)]&= & \left\langle u(t),\left[-\Delta,  i\Gamma_{\psi_R}\right] u(t)\right\rangle+\left\langle  -f(u), i \Gamma_{\psi_R}  u(t)\right\rangle+\left\langle u(t) , i \Gamma_{\psi_R}f(u)  \right\rangle \\
&= & I_1+I_2+I_3,
 \end{eqnarray*}
 where $[X, Y] \equiv X Y-Y X$ denotes the commutator of operators $X$ and $Y$.
According to the localized radial virial estimate in \cite{TBDH2016}, we obtain
\begin{eqnarray*}
I_1&=&\left\langle u(t),\left[-\Delta, i \Gamma_{\psi_R}\right] u(t)\right\rangle \leq 8 \left\|\nabla u\right\|_{L^2}^2+C R^{-2  },\\
  I_2 &=& \left\langle  -f(u), i \Gamma_{\psi_R}  u(t)\right\rangle \\
    &=&-\left\langle  f(u), \left(\nabla \cdot\left(\left(\nabla \psi_R\right) u\right)+\left(\nabla \psi_R\right) \cdot(\nabla u)\right)\right\rangle\\
    & \leq&2 \int_{\mathbb{R}^N}  \nabla \psi_R \cdot \nabla\left(F(u)\right) d x \\
    & =&-2 \int_{\mathbb{R}^N}\left(\Delta \psi_R\right)F(u)dx\\
    & =&- 2(N+1)\int_{\mathbb{R}^N} F(u)dx-2 \int_{|x| \geq  R}(\Delta \psi_R-N)F(u)dx\\
    &\leq&- 2(N+1)\int_{\mathbb{R}^N} F(u)dx+2  \int_{|x| \geq  R}(N-\Delta \psi_R)(\varepsilon|u|^2+C_\varepsilon|u|^{2^*})dx\\
    & \leq&- 2(N+1)\int_{\mathbb{R}^N} F(u)dx+2CR^{-2}\|N-\Delta \psi_R\|_{L^\infty}\|\nabla u\|_{L^2(|x|\geq R)}^2,\\
I_3 &=& \left\langle u(t) , i \Gamma_{\psi_R}f(u)  \right\rangle \\
    &=&-\left\langle  u(t), \left(\nabla \cdot\left(\left(\nabla \psi_R\right) f(u)\right)+\left(\nabla \psi_R\right) \cdot(\nabla f(u))\right)\right\rangle\\
    & \leq&2 \int_{\mathbb{R}^N}  \nabla \psi_R \cdot \nabla\left(F(u)\right) d x \\
    & =&-2 \int_{\mathbb{R}^N}\left(\Delta \psi_R\right)F(u)dx\\
    & =&- 2(N+1)\int_{\mathbb{R}^N} F(u)dx-2 \int_{|x| \geq  R}(\Delta \psi_R-N)F(u)dx\\
    & \leq&- 2(N+1)\int_{\mathbb{R}^N} F(u)dx+2C \int_{|x| \geq  R}(N-\Delta \psi_R)(|u|^2+|u|^{2^*})dx\\
    &\leq&- 2(N+1)\int_{\mathbb{R}^N} F(u)dx+2CR^{-2}\|N-\Delta \psi_R\|_{L^\infty}\|\nabla u\|_{L^2(|x|\geq R)}^2 ,
\end{eqnarray*}
Therefore,
\begin{equation*}
  \frac{d}{d t} \mathcal{M}_{\psi_R}[u(t)] \leq8 \left\|\nabla u\right\|_{L^2}^2+C R^{-2  }- 4(N+1)\int_{\mathbb{R}^N} F(u)dx+4CR^{-2}\|N-\Delta \psi_R\|_{L^\infty}\|\nabla u\|_{L^2(|x|\geq R)}^2,
\end{equation*}
where the constant $C>0$ is independent of $R$. When $R>1$ is sufficiently large, then
$$
\frac{d}{d t} \mathcal{M}_{\psi_R}[u(t)] \leq 8 E(u(t))=8   E(u_0) .
$$
\end{proof}

\begin{proof}[\bf Proof of Theorem \ref{t1.2}]
 By lemma \ref{L3.1}, for a radial initial function $\varphi \in H_{ra d}^1$, \eqref{eq1.1} admits a local solution $u \in H_{r a d}^1$. If $T^*<\infty$, then we are done. If $T^*=\infty$, we show \eqref{eq4.1}. We suppose that $u(t)$ exists for all times $t \geq 0$, i.e. $T^*=\infty$.
It follows from Lemma \ref{L4.1} and conservation of mass, for $R \gg 1$ large enough,
$$
\frac{d}{d t} \mathcal{M}_{\psi_R}[u(t)] \leq 8   E(u_0):=-A^*<0, \quad t \geq 0 .
$$
From this, we infer that
$$
\mathcal{M}_{\psi_R}[u(t)] \leq-A^* t+\mathcal{M}_{\psi_R}[u_0], \quad t \geq 0 .
$$
On the one hand, let $T_0=\frac{2\left|\mathcal{M}_{\psi_R}[u_0]\right|}{A^*}>0$, then for any $t \geq T_0$, we have
 \begin{equation}\label{eq4.4}
   \mathcal{M}_{\psi_R}[u(t)] \leq-\frac{1}{2} A^* t<0 .
 \end{equation}
On the other hand, by Lemma \ref{L2.4} and the conservation of mass, we see that for any $t \in[0,+\infty)$,
\begin{eqnarray*}
\left|\mathcal{M}_{\psi_R}[u(t)]\right| & \leq& C\left(\psi_R\right)\left(\left\||\nabla|^{\frac{1}{2}} u(t)\right\|_{L^2}^2+\|u(t)\|_{L^2}\left\||\nabla|^{\frac{1}{2}} u(t)\right\|_{L^2}\right) \\
& \leq& C\left(\psi_R\right)\left(\left\|\nabla u\right\|_{L^2} \|u\|_{L^2}+\|u(t)\|_{L^2}^{\frac{3}{2}}\left\|\nabla u\right\|_{L^2}^{\frac{1}{2}}\right) \\
& \leq& C\left(\psi_R\right) \left\|\nabla u\right\|_{L^2}^2,
\end{eqnarray*}
where we have used the interpolation estimate
$$
\left\||\nabla|^{\frac{1}{2}} u\right\|_{L^2}^2 \leq C\left\|\nabla u\right\|_{L^2} \|u\|_{L^2}  .
$$
This combined with \eqref{eq4.4} yields that for any $t \geq T_0$,
$$
A^* t \leq-2 \mathcal{M}_{\psi_R}[u(t)] \leq C\left\|\nabla u\right\|_{L^2}^2.
$$
This shows that
$$
\left\|\nabla u(t)\right\|_{L^2}^2 \geq C t , \quad t \geq T_0 .
$$
It means that
$$
\sup\limits_{t \geq 0}\left\|\nabla u(\cdot, t)\right\|_{L^2}=\infty .
$$
\end{proof}

\section{Perturbation theory}
In this section, we will study the long-time perturbation theory.

\begin{proposition}(Perturbation theory)\label{P5.1}
Let $\tilde{u}: I \times \mathbb{R}^N \rightarrow \mathbb{C}$ be a solution to the perturbed Schr\"odinger
equation with general nonlinearity
$$
i \partial_t \tilde{u}+\Delta \tilde{u}+f(\tilde{u})=e
$$
for some function $e$. Suppose that
\begin{equation}\label{eq5.1}
  \|\tilde{u}\|_{L^{\infty} H^1\left(I \times \mathbb{R}^N\right)}  \leq E,
\end{equation}
\begin{equation}\label{eq5.2}
  \|\tilde{u}\|_{L^{\frac{2(N+2)}{N-2}}L^{\frac{2(N+2)}{N-2}} (I \times \mathbb{R}^N)}   \leq L
\end{equation}
for some $E, L>0$. Let $u_0 \in H^1(\mathbb{R}^N)$ with $\left\|u_0\right\|_{L^2(\mathbb{R}^N)} \leq M$ for some $M>0$ and let $t_0 \in I$. There exists $\varepsilon_0=\varepsilon_0(E, L, M)>0$ such that if
\begin{equation}\label{eq5.3}
  \left\|u_0-\tilde{u}\left(t_0\right)\right\|_{H^1}  \leq \varepsilon,
\end{equation}
\begin{equation}\label{eq5.4}
  \left\|e\right\|_{L_I^\infty H^1\cap L_I^2 L^{\frac{2 N}{N+2}}}   \leq \varepsilon
\end{equation}
for $0<\varepsilon<\varepsilon_0$, then the unique global solution $u$ to \eqref{eq1.1} with $u\left(t_0\right)=u_0$ satisfies
\begin{equation}\label{eq5.5}
  \left\|u-\tilde{u}\right\|_{L_I^\infty H^1\cap L_I^\frac{2(N+2)}{N-2}W^{1,{\frac{2N(N+2)}{N^2+4}}}} \leq C(E, L, M) \varepsilon,
\end{equation}
where $C(E, L, M)$ is a non-decreasing function of $E, L$, and $M$.
\end{proposition}
\begin{proof}
Without loss of generality, we may assume $t_0=0 \in I$.  From Theorem \ref{t1.1} we know that $u$ exists globally and
\begin{equation*}
  \|u\|_{L_I^\frac{2(N+2)}{N-2}W^{1,{\frac{2N(N+2)}{N^2+4}}}}\leq\rho,
\end{equation*}
so we need to get \eqref{eq5.5}.

Let
\begin{equation*}
   w:=u-\tilde{u}  \ \text{and}\  A(t):=\|w\|_{L_{(I \cap[-T, T])}^\infty H^1\cap L_{(I \cap[-T, T])}^\frac{2(N+2)}{N-2}W^{1,{\frac{2N(N+2)}{N^2+4}}}}.
\end{equation*}
Note that
\begin{equation*}
  i \partial_t \tilde{w}+\Delta \tilde{w}-f(\tilde{u})+f(\tilde{u}+w) =-e,
\end{equation*}
which is equivalent to the integral equation
$$
w(t)=e^{i t \Delta} w_0+i\int_0^t e^{i\left(t-t^{\prime}\right) \Delta} [-f(\tilde{u})+f(\tilde{u}+w)] d t^{\prime}+i\int_0^t e^{i\left(t-t^{\prime}\right) \Delta}e d t^{\prime}.
$$
Then by Strichartz, H\"older, \eqref{eq5.1}, \eqref{eq5.2}, \eqref{eq5.3}, \eqref{eq5.4}, we obtain
\begin{eqnarray*}\
&&\left\| w(t)\right\|_{L_{(I \cap[-T, T])}^\infty H^1\cap L_{(I \cap[-T, T])}^\frac{2(N+2)}{N-2}W^{1,{\frac{2N(N+2)}{N^2+4}}}}\nonumber\\
&=&\left\|e^{i t \Delta} w_0+i\int_0^t e^{i\left(t-t^{\prime}\right) \Delta} [-f(\tilde{u})+f(\tilde{u}+w)] d t^{\prime}+i\int_0^t e^{i\left(t-t^{\prime}\right) \Delta}e d t^{\prime}\right\|_{L_T^\infty H^1\cap L_T^\frac{2(N+2)}{N-2}W^{1,{\frac{2N(N+2)}{N^2+4}}}}\nonumber\\
&\leq&C\|w_0\|_{H^1}+ \left\|\int_0^t e^{i\left(t-t^{\prime}\right) \Delta} e d t^{\prime}\right\|_{L_T^\infty H^1}+\left\|\int_0^t e^{i\left(t-t^{\prime}\right) \Delta} w d t^{\prime}\right\|_{L_T^\infty H^1}\\
&&+\left\|\int_0^t e^{i\left(t-t^{\prime}\right) \Delta}[ -\tilde{u}\cdot |\tilde{u}|^{\frac{4}{N-2}} +(\tilde{u}+w)\cdot |\tilde{u}+w|^{\frac{4}{N-2}} ] d t^{\prime}\right\|_{L_T^\frac{2(N+2)}{N-2}W^{1,{\frac{2N(N+2)}{N^2+4}}}} \nonumber\\
&\leq&C\|w_0\|_{H^1}+\left\|w\right\|_{L_T^1H^1}+\left\|e\right\|_{L_T^1H^1}+C\left\||w|\cdot( |\tilde{u}+w|^ {\frac{4}{N-2}}+|\tilde{u}|^ {\frac{4}{N-2}})  \right\|_{L_T^2 W^{1,{\frac{2 N}{N+2}}}} \nonumber\\
&\leq&C\|w_0\|_{H^1}+T\left\|w\right\|_{L_T^\infty H^1}+T\left\|e\right\|_{L_T^\infty H^1}+C\left\||w|\cdot  |\tilde{u}+w|^ {\frac{4}{N-2}}  \right\|_{L_T^2 W^{1,{\frac{2 N}{N+2}}}}  \nonumber\\
&&+C\left\||w|\cdot |\tilde{u}|^ {\frac{4}{N-2}}   \right\|_{L_T^2 W^{1,{\frac{2 N}{N+2}}}} \nonumber\\
&\leq&C\|w_0\|_{H^1}+T\left\|w\right\|_{L_T^\infty H^1}+T\left\|e\right\|_{L_T^\infty H^1}+C\left\|w\right\|_{L_T^{\frac{2(N+2)}{N-2}}W^{1,{\frac{2N(N+2)}{N^2+4}}}}\cdot\left\| \tilde{u}  \right\|_{L_T^{\frac{2(N+2)}{N-2}}L^{\frac{2(N+2)}{N-2}}}^{\frac{4}{N-2}}\nonumber\\
&&+C\left\|w\right\|_{L_T^{\frac{2(N+2)}{N-2}}W^{1,{\frac{2N(N+2)}{N^2+4}}}}\cdot\left\| \tilde{u} +w \right\|_{L_T^{\frac{2(N+2)}{N-2}}L^{\frac{2(N+2)}{N-2}}}^{\frac{4}{N-2}}\\
&\leq&C\|w_0\|_{H^1}+T\left\|w\right\|_{L_T^\infty H^1}+T\left\|e\right\|_{L_T^\infty H^1}+C\left\|w\right\|_{L_T^{\frac{2(N+2)}{N-2}}W^{1,{\frac{2N(N+2)}{N^2+4}}}}\cdot\left\| \tilde{u}  \right\|_{L_T^{\frac{2(N+2)}{N-2}}L^{\frac{2(N+2)}{N-2}}}^{\frac{4}{N-2}}\nonumber\\
&&+C\left\|w\right\|_{L_T^{\frac{2(N+2)}{N-2}}W^{1,{\frac{2N(N+2)}{N^2+4}}}}\cdot\rho^{\frac{4}{N-2}}\\
&\leq& C\varepsilon+TA(t)+T\varepsilon+CA(t)L^{\frac{4}{N-2}}+CA(t)\rho^{\frac{4}{N-2}},
\end{eqnarray*}
where all space time norms are over $(I \cap[-T, T]) \times \mathbb{R}^N$. Using the standard continuity argument to remove the restriction to $[-T, T]$, we derive \eqref{eq5.5}.
\end{proof}

\begin{remark}\label{r4.1}
Note that $f(u) \in L_T^\infty H^1\cap L_T^\frac{2(N+2)}{N-2}W^{1,{\frac{2N(N+2)}{N^2+4}}}$ and hence
\begin{equation*}
   \left\|\int_t^{\infty} e^{i\left(t-t^{\prime}\right) \Delta} f(u) d t^{\prime}\right\|_{H^1} \rightarrow 0,\  t \rightarrow+\infty.
\end{equation*}
Then, $u(t)=e^{i(t-a) \Delta} u_0+\int_a^t e^{i\left(t-t^{\prime}\right) \Delta} f(u) d t^{\prime}$ and hence $u^{+}=e^{-i a \Delta} u_0+\int_a^{\infty} e^{-i t^{\prime} \Delta} f(u) d t^{\prime}$ has the desired property. In fact note that the argument used at the beginning of the proof of Proposition \ref{P5.1} shows that it suffices to assume $u$ to be a solution of \eqref{eq1.1} in $I^{\prime} \times \mathbb{R}^N, I^{\prime} \Subset I$, such that $\|u\|_{L_{I}^{\frac{2(N+2)}{N-2}}W^{1,{\frac{2N(N+2)}{N^2+4}}}}<\infty$.
\end{remark}

\section{Some variational estimates}
In this section, let us first recall some variational results.
If we look for standing wave solutions(form $\Phi(t, x)=e^{-imt} u(x)$) about the following equation
\begin{equation}\label{eq5.6}
  i \partial_t \Phi+\Delta \Phi +f(\Phi)=0,\ (x, t) \in \mathbb{R}^N \times \mathbb{R},
\end{equation}
then \eqref{eq5.6} becomes non-linear elliptic
equation
\begin{equation}\label{eq5.7}
 \Delta u +f(u)=0,\ (x, t) \in \mathbb{R}^N \times \mathbb{R}.
\end{equation}
As described in \cite{HB1983}, if we add condition $(F_6)$ under the assumptions of $(F_1)-(F_5)$, that is,
\begin{itemize}
\item[$(F_6)$]\ There exists $\xi>0$ such that $F(\xi)=\int_0^\xi f(s) d s>0$.
\end{itemize}
then we obtain the existence of a solution of \eqref{eq5.7}, i.e., \eqref{eq5.7} possesses a solution $u$ such that

i) $u>0$ on $\mathbb{R}^N$.

ii) $u$ is spherically symmetric: $u(x)=u(r)$, where $r=|x|$, and $u$ decreases with respect to $r$.

iii) $u \in C^2(\mathbb{R}^N)$.

iv) $u$ together with its derivatives up to order 2 have exponential decay at infinity:
$$
\left|D^\alpha u(x)\right| \leqq C e^{-\delta|x|}, \quad x \in \mathbb{R}^N,
$$
for some $C, \delta>0$ and for $|\alpha| \leqq 2$.

Now, we assume that the solution of \eqref{eq5.7} is $W(x)$. The equation \eqref{eq5.7} gives
\begin{equation*}
  \int_{\mathbb{R}^N}|\nabla W|^2dx=\int_{\mathbb{R}^N}F(u)dx,\ F(u)=\int_0^uf(s)ds.
\end{equation*}
Moreover, note that $W(x)$ is a solution of \eqref{eq5.7}, so we have the following Pohozaev identity
 \begin{equation*}
   \frac{N-2}{2} \int_{\mathbb{R}^N}|\nabla W|^2 d x=N \int_{\mathbb{R}^N} F(W) d x.
 \end{equation*}
Hence,
\begin{equation*}
E(W)=\frac{1}{2}\int_{\mathbb{R}^N}|\nabla W|^2dx-\int_{\mathbb{R}^N} F(W) d x =\frac{1}{N}\int_{\mathbb{R}^N}|\nabla W|^2dx .
\end{equation*}
\begin{lemma}\label{L5.1}
Assume $u$  satisfies
$$
\|\nabla u\|_{L^2}^2<\|\nabla W\|_{L^2}^2 .
$$
Moreover, let $E(u)\leq(1-\delta_0)E(W)$, where $\delta_0>0$. Then, there exists $\delta_1>0,\ \bar{\delta}>0$ such that
\begin{equation}\label{eq5.8}
  \int_{\mathbb{R}^N}|\nabla u|^2dx \leq(1-\delta_1) \int_{\mathbb{R}^N}|\nabla W|^2dx,
\end{equation}
\begin{equation}\label{eq5.9}
  \int_{\mathbb{R}^N}(|\nabla u|^2-F(u))dx \geq \bar{\delta} \int_{\mathbb{R}^N}|\nabla u|^2dx,
\end{equation}
\begin{equation}\label{eq5.98}
  \int_{\mathbb{R}^N}(|\nabla u|^2-f(u)\overline{u})dx \geq \bar{\delta} \int_{\mathbb{R}^N}|\nabla u|^2dx,
\end{equation}
\begin{equation}\label{eq5.10}
  E(u) \geq 0 .
\end{equation}
\end{lemma}
\begin{proof}
In order to get \eqref{eq5.9}, by $(F_1)-(F_2)$, we have, for any $\varepsilon>0$,
\begin{equation*}
F(u)\leq \varepsilon|u|^2+C_{\varepsilon}|u|^{2^*}.
\end{equation*}
Hence, it holds
\begin{eqnarray*}
  \int_{\mathbb{R}^N}(|\nabla u|^2-F(u))dx &\geq&  \int_{\mathbb{R}^N}|\nabla u|^2dx-\varepsilon\|u\|_{L^2}^2-C_{\varepsilon}\|u\|_{L^{2^*}}^{2^*} \\
    &=&\|\nabla u\|_{L^2}^2-\varepsilon\|\nabla u\|_{L^2}^2-C_{\varepsilon}\|\nabla u\|_{L^2}^{2^*}\\
    &:=& \bar{\delta}\int_{\mathbb{R}^N}|\nabla u|^2dx,
\end{eqnarray*}
which implies \eqref{eq5.9} holds. Similarly, it can be proven that \eqref{eq5.11} is established. To prove \eqref{eq5.8}, by $E(u)\leq(1-\delta_0)E(W)$ and $\|\nabla u\|_{L^2}^2<\|\nabla W\|_{L^2}^2$, we have
\begin{eqnarray*}
  \frac{1}{2}\int_{\mathbb{R}^N}|\nabla u|^2dx-\int_{\mathbb{R}^N}F(u)dx  \leq  (1-\delta_0)E(W)=\frac{1-\delta_0}{2N}\int_{\mathbb{R}^N}|\nabla W|^2dx,
\end{eqnarray*}
so
\begin{eqnarray*}
  \int_{\mathbb{R}^N}|\nabla u|^2dx &\leq&\frac{1}{2}\int_{\mathbb{R}^N}|\nabla u|^2dx+ \frac{1-\delta_0}{2N}\int_{\mathbb{R}^N}|\nabla W|^2dx+\int_{\mathbb{R}^N} F(u) d x  \\
    &\leq&\frac{1}{2}\int_{\mathbb{R}^N}|\nabla W|^2dx+ \frac{1-\delta_0}{2N}\int_{\mathbb{R}^N}|\nabla W|^2dx  +\varepsilon\|\nabla u\|_{L^2}^2+C_{\varepsilon}\|\nabla u\|_{L^2}^{2^*}\\
    &\leq&\frac{1}{2}\int_{\mathbb{R}^N}|\nabla W|^2dx+ \frac{1-\delta_0}{2N}\int_{\mathbb{R}^N}|\nabla W|^2dx  +\varepsilon\|\nabla W\|_{L^2}^2+C_{\varepsilon}\|\nabla W\|_{L^2}^{\frac{2^*}{2}}\\
    &:=&(1-\delta_1)\int_{\mathbb{R}^N}|\nabla W|^2dx.
\end{eqnarray*}
According to the proof of \eqref{eq5.9}, \eqref{eq5.10} obviously holds, this completes the proof.
\end{proof}
\begin{remark}\label{r6.1}
From Lemma \ref{L5.1}, we know that the selection of $W$ is not arbitrary. In fact, we need to choose $W$ such that $\int_{\mathbb{R}^N}|\nabla W|^2dx$ is small. Moreover, according to the conditions of the nonlinear term and $f$ is odd, there are actually infinite standing wave solutions for equation \eqref{eq1.1}, see \cite{HBII1983}. In this case, we only need to take the solution that minimizes $\int_{\mathbb{R}^N}|\nabla W|^2dx$.
\end{remark}
\begin{corollary}\label{c5.1}
Assume that $u\in H^1$ and $\int_{\mathbb{R}^N}\left|\nabla u\right|^2dx<\int_{\mathbb{R}^N}|\nabla W|^2dx$, then $E(u)\geq0$.
\end{corollary}
\begin{proof}
If $E(u)\geq E(W)=\frac{1}{N}\int_{\mathbb{R}^N}|\nabla W|^2dx$, this is obvious. If $E(u)< E(W)$, the claim follows from \eqref{eq5.10}.
\end{proof}

\begin{theorem}\label{t5.1}(Energy trapping). Let $u$ be a solution of the \eqref{eq1.1}, with $t_0=0,\left.u\right|_{t=0}=u_0$ such that for $\delta_0>0$,
$$
\int_{\mathbb{R}^N}\left|\nabla u_0\right|^2dx<\int_{\mathbb{R}^N}|\nabla W|^2dx,\ E(u_0)<(1-\delta_0)E(W).
$$
Let $I(0\in I)$ be the maximal interval of existence. Let $\delta_0,\ \bar{\delta}$ be as in Lemma \ref{L5.1}. Then, for each $t \in I$, we have
\begin{equation}\label{eq5.11}
  \int_{\mathbb{R}^N}|\nabla u(t)|^2dx \leq(1-\delta_1) \int_{\mathbb{R}^N}|\nabla W|^2dx,
\end{equation}
\begin{equation}\label{eq5.12}
  \int_{\mathbb{R}^N}(|\nabla u(t)|^2-F(u(t)))dx \geq  \bar{\delta} \int_{\mathbb{R}^N}|\nabla u(t)|^2dx,
\end{equation}
\begin{equation}\label{eq5.100}
  \int_{\mathbb{R}^N}(|\nabla u(t)|^2-f(u(t))\overline{u}(t))dx \geq  \bar{\delta} \int_{\mathbb{R}^N}|\nabla u(t)|^2dx,
\end{equation}
\begin{equation}\label{eq5.13}
  E(u(t))  \geq  0 .
\end{equation}
\end{theorem}
\begin{proof} By energy conservation, $E(u(t))=E(u_0),\ t \in I$ and the theorem follows directly from Lemma \ref{L5.1} and a continuity argument.
\end{proof}
\begin{corollary}\label{c5.2}
Let $u, u_0$ be as in Theorem \ref{t5.1}. Then for all $t \in I$ we have $E(u(t)) \simeq \int_{\mathbb{R}^N}|\nabla u(t)|^2dx \simeq \int_{\mathbb{R}^N}\left|\nabla u_0\right|^2dx$, with comparability constants which depend only on $N$.
\end{corollary}
\begin{proof}
$E(u(t)) \leq \int_{\mathbb{R}^N}|\nabla u(t)|^2dx$, but by the proof of \eqref{eq5.9} there exists $\widetilde{\delta}$ such that
\begin{eqnarray*}
E(u(t)) &=&\frac{1}{2} \int_{\mathbb{R}^N}\left|\nabla u(t, x)\right|^2 d x-  \int_{\mathbb{R}^N}F(u(t,x))  d x\\
&\geq& \widetilde{\delta} \int_{\mathbb{R}^N}|\nabla u(t)|^2dx,
\end{eqnarray*}
so the first equivalence follows. For the second one note that
\begin{equation*}
  E(u(t))= E(u_0) \simeq \int_{\mathbb{R}^N}\left|\nabla u_0\right|^2dx,
\end{equation*}
by the first equivalence when $t=0$.
\end{proof}

\section{Concentration compactness}
The main purpose of this section is to prove the following concentration compactness lemma, which plays a crucial role in the proof of the theorem in the next section.

\begin{lemma}\label{L5.4}(Concentration compactness)
Let $\left\{v_{0, n}\right\} \in H^1, \left\|v_{0, n}\right\|_{H^1}$ $\leq\rho$ . Assume that $\left\|e^{i t \Delta} v_{0, n}\right\|_{L_{t,x}^{\frac{2(N+2)}{N-2}}} \geq \rho>0$, where $\rho$ is as in Lemma \ref{L3.1}. Then there exists a sequence $\left\{V_{0, j}\right\}_{j=1}^{\infty}$ in $H^1$, a subsequence of $\left\{v_{0, n}\right\}$$($which we still call $\left\{v_{0, n}\right\})$  and a triple $\left(\lambda_{j, n} ; x_{j, n} ; t_{j, n}\right) \in \mathbb{R}^{+} \times \mathbb{R}^N \times \mathbb{R}$, with
$$
\frac{\lambda_{j, n}}{\lambda_{j^{\prime}, n}}+\frac{\lambda_{j^{\prime}, n}}{\lambda_{j, n}}+\frac{\left|t_{j, n}-t_{j^{\prime}, n}\right|}{\lambda_{j, n}^2}+\frac{\left|x_{j, n}-x_{j^{\prime}, n}\right|}{\lambda_{j, n}} \rightarrow \infty
$$
as $n \rightarrow \infty$ for $j \neq j^{\prime}$$($we say that $(\lambda_{j, n} ; x_{j, n} ; t_{j, n})$ is orthogonal if this property is verified$)$ such that
\begin{equation}\label{eq5.14}
  \left\|V_{0,1}\right\|_{H^1} \geq \alpha_0(\rho)>0 .
\end{equation}

If $V_j^l(x, t)=e^{i t \Delta} V_{0, j}$, then, given $\varepsilon_0>0$, there exists $J=J\left(\varepsilon_0\right)$ and
\begin{equation}\label{eq5.15}
  \left\{w_n\right\}_{n=1}^{\infty} \in H^1\ \text{so that}\  v_{0, n}=\sum_{j=1}^J \frac{1}{\lambda_{j, n}^{\frac{N-2}{2}}} V_j^l\left(\frac{x-x_{j, n}}{\lambda_{j, n}}, \frac{-t_{j, n}}{\lambda_{j, n}^2}\right)+w_n
\end{equation}
with $\left\|e^{i t \Delta} w_n\right\|_{L_{(-\infty, +\infty)}^{\frac{2(N+2)}{N-2}}W^{1,{\frac{2N(N+2)}{N^2+4}}}} \leq \varepsilon_0$, for n large
\begin{equation}\label{eq5.16}
  \int_{\mathbb{R}^N}\left|\nabla v_{0, n}\right|^2dx=\sum_{j=1}^J \int_{\mathbb{R}^N}\left|\nabla V_{0, j}\right|^2dx+\int_{\mathbb{R}^N}\left|\nabla w_n\right|^2dx+o(1),\ n \rightarrow \infty,
\end{equation}
\begin{equation}\label{eq5.17}
   E\left(v_{0, n}\right)=\sum_{j=1}^J E (V_j^l (\frac{-t_{j, n}} { \lambda_{j, n}^2} ) )+E\left(w_n\right)+o(1),\ n \rightarrow \infty.
\end{equation}
\end{lemma}
The proof of this lemma originates from Keraani \cite{KSK2001}, but we need to modify the proof since this paper considers general nonlinear terms. Firstly, we consider linear equation
\begin{equation}\label{eq7.5}
  \left\{\begin{array}{l}
i \partial_t u+\Delta u  =0,\ (x, t) \in \mathbb{R}^N \times \mathbb{R}, \\
\left.u(0,x)\right|_{t=0}=u_0(x) \in  H ^1(\mathbb{R}^N).
\end{array}\right.
\end{equation}

\begin{lemma}\label{L7.2}
Let $(\varphi_n)_{n \geq  0}$ be a bounded sequence in $H^1(\mathbb{R}^N)$. Let $(v_n)_{n \geqslant 0}$ be the sequence of solutions to \eqref{eq7.5} with initial data $v_n(0, x)=\varphi_n(x)$. Then there exist a subsequence $(v_n^{\prime})$ of $(v_n)$, a sequence $\left(\mathbf{h}^j\right)_{j \geqslant 1}$ of scales, a sequence $\left(\mathbf{z}^j\right)_{j \geqslant 1}$ of cores and a sequence $\left(V^j\right)_{j \geqslant 1}$ of solutions to \eqref{eq7.5}, such that

(i) the pairs $\left(\mathbf{h}^j, \mathbf{z}^j\right)$ are pairwise orthogonal;

(ii) for every $l \geq 1$,
 \begin{equation}\label{eq7.6}
   v_n^{\prime}(x, t)=\sum_{j=1}^l \frac{1}{(h_n^j)^{\frac{N-2}{2}}} V^j\left( \frac{x-x_n^j}{h_n^j},\frac{t-t_n^j}{\left(h_n^j\right)^2}\right)+w_n^l(x, t),
 \end{equation}
with
 \begin{equation}\label{eq5.99}
   \limsup\limits_{n \rightarrow \infty}\left\|w_n^l\right\|_{L^q\left(\mathbb{R}, L^r\left(\mathbb{R}^N\right)\right)} \rightarrow 0, \ l \rightarrow \infty
 \end{equation}
for every $H^1$-admissible pair $(q, r)$(defined in Definition \ref{D2.3}), and, for every $l \geqslant 1$,
 \begin{equation}\label{eq7.7}
   \int_{\mathbb{R}^N}\left|\nabla v_{ n}'\right|^2dx=\sum_{j=1}^l \int_{\mathbb{R}^N}\left|\nabla V^{j}\right|^2dx+\int_{\mathbb{R}^N}\left|\nabla w_n^l\right|^2dx +o(1), \quad n \rightarrow \infty .
 \end{equation}
\end{lemma}
\begin{proof}
As stated by \cite{KCEMF2006}. It is based on the ``refined Sobolev inequality'' $(N=3)$
$$
\|h\|_{L^6\left(\mathbb{R}^3\right)} \leq C\|\nabla h\|_{L^2\left(\mathbb{R}^3\right)}^{1 / 3}\|\nabla h\|_{\dot{B}_{2, \infty}^0}^{2 / 3},
$$
where $\dot{B}_{2, \infty}^0$ is the standard Besov space \cite{JBJL1976}. The proof of this lemma is completely similar to \cite{{KCEMF2006},{KSK2001}}, so we omit it here.
\end{proof}

\begin{lemma}\label{L7.3}
Let $\left(v_n^{\prime}\right)$ be the subsequence of $(v_n)$ given by Lemma \ref{L7.2}. Let $\left(u_n^{\prime}\right)$ be the sequence of solutions to \eqref{eq1.1}, with the same Cauchy data at $t=0$ as $v_n^{\prime}$. Then
 \begin{equation}\label{eq7.8}
u_n^{\prime}(x, t)=\sum_{j=1}^l \frac{1}{(h_n^j)^{\frac{N-2}{2}}} U^j\left(\frac{x-x_n^j}{h_n^j},\frac{t-t_n^j}{\left(h_n^j\right)^2}  \right)+w_n^l(x, t)+r_n^l(x, t),
 \end{equation}
with
$$
\limsup _{n \rightarrow \infty}\left(\sup _{t \in \mathbb{R}} E\left(r_n^l(t, \cdot)\right)+\left\|r_n^l\right\|_{L_{t,x}^{\frac{2(N+2)}{N-2}}\left(\mathbb{R}\times\mathbb{R}^N\right)}+\left\|\nabla r_n^l\right\|_{L_{t,x}^{\frac{2(2+N)}{N}}\left(\mathbb{R}\times\mathbb{R}^N\right)}\right) \rightarrow 0,\ l \rightarrow \infty,
$$
where $V^j, \mathbf{h}^j, \mathbf{z}^j, w_n^l$ are as in \eqref{eq7.6} and $U^j$ is the nonlinear profile associated to $\left(V^j, \mathbf{h}^j, \mathbf{z}^j\right)$(defined in Definition \ref{D2.4}).
\end{lemma}
\begin{proof}
Let us first restate the problem. Let $\left(\varphi_n\right)_{n \geqslant 0}$ be a bounded sequence in $H^1(\mathbb{R}^N)$, such that $\limsup\limits_{n \rightarrow \infty}\left\|u_n(0, .)\right\|_{H^1(\mathbb{R}^N)}<\rho$ (where $\rho$ is given by Lemma \ref{L3.2}). Let $\left(v_n\right)_{n \geqslant 0}\left(\operatorname{resp} .\left(u_n\right)_{n \geqslant 0}\right)$ the sequence of solutions to \eqref{eq7.5} (resp. to \eqref{eq1.1}) with initial data $\left(\varphi_n\right)_{n \geqslant 0}$. Lemma \ref{L7.2} provides a decomposition of $\left(v_n\right)_{n \geqslant 0}$ for a subsequence $\left(v_n^{\prime}\right)$ of $\left(v_n\right)$ in the form
$$
v_n^{\prime}(x, t)=\sum_{j=1}^l \frac{1}{(h_n^j)^{\frac{N-2}{2}}} V^j\left(\frac{x-x_n^j}{h_n^j},\frac{t-t_n^j}{\left(h_n^j\right)^2}  \right)+w_n^l(x, t),
$$
where $\left(V^j\right)$ is a family of solutions to \eqref{eq7.5}, $\left(\mathbf{h}^j, \mathbf{z}^j\right)_{j \geqslant 1}$ is a pairwise orthogonal family of scales-cores, and the remainder $w_n^l$ satisfies
$$
\limsup _{n \rightarrow \infty}\left\|w_n^l\right\|_{L_{t,x}^{\frac{2(N+2)}{N-2}}\left(\mathbb{R}\times\mathbb{R}^N\right)} \rightarrow  0,\ l \rightarrow \infty .
$$
Also, the following almost orthogonality identity holds
$$
\int_{\mathbb{R}^N}\left|\nabla v_{ n}'\right|^2dx=\sum_{j=1}^l \int_{\mathbb{R}^N}\left|\nabla V^{j}\right|^2dx+\int_{\mathbb{R}^N}\left|\nabla w_n^l\right|^2dx +o(1), \quad n \rightarrow \infty .
$$
Let $\left(u_n^{\prime}\right)$ be the sequence of solutions of \eqref{eq7.5}, with the same Cauchy data at $t=0$ as $v_n^{\prime}$ and $U^j$ the nonlinear profile associated to $\left(V^j, \mathbf{h}^j, \mathbf{z}^j\right)$ for every $j \geqslant 1$. Observe that, in view of \eqref{eq7.7}, $\int_{\mathbb{R}^N}(V^j)dx<\rho^2$ and then the nonlinear profile $U^j$ is globally well defined. We set
$$
r_n^l(x, t)=u_n^{\prime}(x, t)-\sum_{j=1}^l \frac{1}{(h_n^j)^{\frac{N-2}{2}}} U^{(j)}\left(\frac{x-x_n^j}{h_n^j},\frac{t-t_n^j}{\left(h_n^j\right)^2}  \right)-w_n^l(x, t) .
$$
Our purpose is to prove that
$$
\limsup _{n \rightarrow \infty}\left(\sup _{t \in \mathbb{R}} E\left(r_n^l(t, .)\right)+\left\|r_n^l\right\|_{L_{t,x}^{\frac{2(N+2)}{N-2}}\left(\mathbb{R}\times\mathbb{R}^N \right)}+\left\|\nabla r_n^l\right\|_{L_{t,x}^{\frac{2(2+N)}{N}}\left(\mathbb{R}\times\mathbb{R}^N \right)}\right) \rightarrow 0,\ l \rightarrow \infty.
$$
The following notations will be used
\begin{eqnarray*}
&&p(z)  =f(z), \\
&&V_n^j(x, t) = \frac{1}{(h_n^j)^{\frac{N-2}{2}}}  V^j\left(\frac{x-x_n^j}{h_n^j},\frac{t-t_n^j}{\left(h_n^j\right)^2}\right), \\
&&U_n^j(x, t)  = \frac{1}{(h_n^j)^{\frac{N-2}{2}}}  U^j\left( \frac{x-x_n^j}{h_n^j},\frac{t-t_n^j}{\left(h_n^j\right)^2}\right), \\
&&W_n^l = \sum\limits_{j=1}^l U_n^j, \\
&&f_n^l = \sum\limits_{j=1}^l p\left(U_n^j\right)-p\left(\sum_{j=1}^l U_n^j+w_n^l+r_n^l\right) .
\end{eqnarray*}
The function $r_n^l$ provided by \eqref{eq7.8} satisfies the equation
$$
\left\{\begin{array}{l}
i \partial_t r_n^l+  \Delta r_n^l=f_n^l \\
r_n^l(x, 0)=\sum\limits_{j=1}^l\left(V_n^j-U_n^j\right)(x, 0) .
\end{array}\right.
$$
We also introduce the norm
$$
\interleave g\interleave_I=\|\nabla g\|_{L_{t,x}^{\frac{2(N+2)}{N}}\left(I \times \mathbb{R}^N\right)}+\|g\|_{L_{t,x}^{\frac{2(N+2)}{N-2}}\left(I \times \mathbb{R}^N\right)} .
$$
Note that, by Strichartz estimates \eqref{eq2.1} and \eqref{eq2.2}, we get
$$
\interleave e^{it\Delta} \varphi\interleave_{\mathbb{R}} \leq C\|\varphi\|_{H^1\left(\mathbb{R}^N\right)}
$$
for every $\varphi \in H^1(\mathbb{R}^N)$. In the rest of the paper we note, for every $a \in \mathbb{R}$,
 \begin{equation}\label{eq6.99}
   \gamma_n^l(a)=\left\|\nabla r_n^l(\cdot,a)\right\|_{L^2\left(\mathbb{R}^N\right)} .
 \end{equation}
Recall that our purpose is to prove that
$$
\limsup\limits_{n \rightarrow \infty}\left(\sup\limits_{t \in \mathbb{R}} E\left(r_n^l(\cdot, t)\right)+\interleave r_n^l\interleave_{\mathbb{R}}\right) \rightarrow 0,\ l \rightarrow \infty .
$$

The following lemma is a combination of Strichartz estimates and energy inequalities, see \cite{{TCS1993},{KSK2001},{KSK2006}}.
\begin{lemma}\label{L7.4}
Let $v \in C([a, b], \dot{H}^1(\mathbb{R}^N))$ be a solution on $I=[a, b]$ of the nohomogeneous Schr\"odinger equation
$$
i \partial_t v+\Delta v=f,
$$
with $\nabla f \in L^{\frac{2(N+2)}{N+4}}(I \times \mathbb{R}^N)$. Then we have
$$
\interleave v\interleave_I+\sup _{t \in I} \|\nabla v(t))\|_{L^2(\mathbb{R}^N)} \leq C\left(\|\nabla v(a))\|_{L^2(\mathbb{R}^N)}+\|\nabla f\|_{L^{\frac{2(N+2)}{N+4}}\left(I \times \mathbb{R}^N\right)}\right) .
$$
\end{lemma}
Applying Lemma \ref{L7.4} to $r_n^l$ on $I=[a, b]$, it holds
 \begin{equation}\label{eq7.9}
   \interleave r_n^l\interleave_I+\sup\limits_{t \in I} \|\nabla r_n^l(t)\|_{L^2(\mathbb{R}^N)}  \leq C\left(\gamma_n^l(a)+\left\|\nabla f_n^l\right\|_{L^{\frac{2(N+2)}{N+4}}\left(I \times \mathbb{R}^N\right)}\right) .
 \end{equation}
Next, we estimate
\begin{eqnarray}\label{eq7.10}
  &&\left\|\nabla f_n^l\right\|_{L^{\frac{2(N+2)}{N+4}}\left(I \times \mathbb{R}^N\right)} \nonumber\\
  &=& \left\|\nabla\left[ \sum\limits_{j=1}^l p\left(U_n^j\right)-p\left(\sum_{j=1}^l U_n^j+w_n^l+r_n^l\right)\right]\right\|_{L^{\frac{2(N+2)}{N+4}}\left(I \times \mathbb{R}^N\right)} \nonumber\\
    &\leq&\left\|\nabla\left[ \sum\limits_{j=1}^l p\left(U_n^j\right)-p\left(W_n^l\right)\right]\right\|_{L^{\frac{2(N+2)}{N+4}}\left(I \times \mathbb{R}^N\right)}+\left\|\nabla\left(p\left(W_n^l+w_n^l\right)-p\left(W_n^l\right)\right)\right\|_{L^{\frac{2(N+2)}{N+4}}\left(I \times \mathbb{R}^N\right)} \nonumber\\
    &&+\left\|\nabla\left(p\left(W_n^l+w_n^l+r_n^l\right)-p\left(W_n^l+w_n^l\right)\right)\right\|_{L^{\frac{2(N+2)}{N+4}}\left(I \times \mathbb{R}^N\right)}.
\end{eqnarray}
Furthermore, a combination of Leibnitz formula, H\"older inequality and $(F_1)-(F_4)$ gives
\begin{eqnarray}\label{eq7.11}
&&\left\|\nabla\left(p\left(W_n^l+w_n^l+r_n^l\right)-p\left(W_n^l+w_n^l\right)\right)\right\|_{L^{\frac{2(N+2)}{N+4}}\left(I \times \mathbb{R}^N\right)} \nonumber\\
&\leq&C\left\| |\nabla r_n^l|\left(\left(W_n^l+w_n^l+r_n^l\right)^{\frac{4}{N-2}}+\left(W_n^l+w_n^l\right)^{\frac{4}{N-2}}\right) \right\|_{L^{\frac{2(N+2)}{N+4}}\left(I \times \mathbb{R}^N\right)}+ \varepsilon\|\nabla  r_n^l\|_{L^{\frac{2(N+2)}{N+4}}\left(I \times \mathbb{R}^N\right)}  \nonumber\\
&\leq&C\left\| |\nabla r_n^l|\left(W_n^l+w_n^l+r_n^l\right)^{\frac{4}{N-2}}\right\|_{L^{\frac{2(N+2)}{N+4}}\left(I \times \mathbb{R}^N\right)}+\left\||\nabla r_n^l|\left(W_n^l+w_n^l\right)^{\frac{4}{N-2}} \right\|_{L^{\frac{2(N+2)}{N+4}}\left(I \times \mathbb{R}^N\right)} \nonumber\\
&&+ \varepsilon\|\nabla  r_n^l\|_{L^{\frac{2(N+2)}{N+4}}\left(I \times \mathbb{R}^N\right)} \nonumber\\
&\leq&C\left\|\nabla r_n^l\right\|_{L^{\frac{2(N+2)}{N}}\left(I \times \mathbb{R}^N\right)}\cdot\left\| \left(W_n^l+w_n^l+r_n^l\right)^{\frac{4}{N-2}}\right\|_{L^{\frac{2(N+2)}{4}}\left(I \times \mathbb{R}^N\right)}+ \varepsilon\|\nabla  r_n^l\|_{L^{\frac{2(N+2)}{N+4}}\left(I \times \mathbb{R}^N\right)} \nonumber\\
&&+\left\||\nabla r_n^l|\left(W_n^l+w_n^l\right)^{\frac{4}{N-2}} \right\|_{L^{\frac{2(N+2)}{N+4}}\left(I \times \mathbb{R}^N\right)} \nonumber\\
&\leq&C\left\|\nabla r_n^l\right\|_{L^{\frac{2(N+2)}{N}}\left(I \times \mathbb{R}^N\right)}\cdot\left\| W_n^l+w_n^l+r_n^l\right\|_{L^{\frac{2(N+2)}{N-2}}\left(I \times \mathbb{R}^N\right)}^{\frac{4}{N-2}}+ \varepsilon\|\nabla  r_n^l\|_{L^{\frac{2(N+2)}{N+4}}\left(I \times \mathbb{R}^N\right)} \nonumber\\
&&+\left\||\nabla r_n^l|\left(W_n^l+w_n^l\right)^{\frac{4}{N-2}} \right\|_{L^{\frac{2(N+2)}{N+4}}\left(I \times \mathbb{R}^N\right)} \nonumber\\
&\leq&C\left\|\nabla r_n^l\right\|_{L^{\frac{2(N+2)}{N}}\left(I \times \mathbb{R}^N\right)}\cdot \sum_{\alpha=1}^{\frac{4}{N-2}}\left\|W_n^l+w_n^l\right\|_{L^{\frac{2(N+2)}{N-2}}\left(I \times \mathbb{R}^N\right)}^{\frac{4}{N-2}-\alpha}\cdot\left\|(r_n^l)\right\|_{L^{\frac{2(N+2)}{N-2}}\left(I \times \mathbb{R}^N\right)}^\alpha \nonumber\\
&&+\left\||\nabla r_n^l|\left(W_n^l+w_n^l\right)^{\frac{4}{N-2}} \right\|_{L^{\frac{2(N+2)}{N+4}}\left(I \times \mathbb{R}^N\right)}+ \varepsilon\|\nabla  r_n^l\|_{L^{\frac{2(N+2)}{N+4}}\left(I \times \mathbb{R}^N\right)} \nonumber\\
&=&C \sum_{\alpha=1}^{\frac{4}{N-2}}\interleave W_n^l+w_n^l\interleave_{I}^{\frac{4}{N-2}-\alpha}\cdot\interleave r_n^l\interleave_{I}^{\alpha+1}+C\left\||\nabla r_n^l|\left(W_n^l+w_n^l\right)^{\frac{6-N}{N-2}}\left(W_n^l+w_n^l\right) \right\|_{L^{\frac{2(N+2)}{N+4}}\left(I \times \mathbb{R}^N\right)} \nonumber\\
&&+ \varepsilon\|\nabla  r_n^l\|_{L^{\frac{2(N+2)}{N+4}}\left(I \times \mathbb{R}^N\right)} \nonumber\\
& \leq& C \sum_{\alpha=1}^{\frac{4}{N-2}}\interleave W_n^l+w_n^l\interleave_{I}^{\frac{4}{N-2}-\alpha} \interleave r_n^l\interleave_{I}^{\alpha+1}+C\interleave r_n^l\interleave_I \interleave W_n^l+w_n^l\interleave_I^{\frac{6-N}{N-2}}\left\|W_n^l+w_n^l) \right\|_{L^{\frac{2(N+2)}{N-2}}\left(I \times \mathbb{R}^N\right)}\nonumber\\
&&+o(1).
\end{eqnarray}
We set
 \begin{equation}\label{eq7.12}
   \delta_n^l=\left\|\nabla\left[ \sum\limits_{j=1}^l p\left(U_n^j\right)-p\left(W_n^l\right)\right]\right\|_{L^{\frac{2(N+2)}{N+4}}\left(\mathbb{R} \times \mathbb{R}^N\right)}+\left\|\nabla\left(p\left(W_n^l+w_n^l\right)-p\left(W_n^l\right)\right)\right\|_{L^{\frac{2(N+2)}{N+4}}\left(\mathbb{R} \times \mathbb{R}^N\right)}.
 \end{equation}
Combining \eqref{eq7.9}, \eqref{eq7.10}, \eqref{eq7.11} and \eqref{eq7.12}, we have
 \begin{eqnarray}\label{eq7.13}
&&\interleave r_n^l\interleave_I+\sup\limits_{t \in I} \|\nabla r_n^l(t)\|_{L^2(\mathbb{R}^N)}\nonumber\\
&\leq& C\left(\gamma_n^l(a)+\left\|\nabla f_n^l\right\|_{L^{\frac{2(N+2)}{N+4}}\left(I \times \mathbb{R}^N\right)}\right)  \nonumber\\
&\leq&C\sum_{\alpha=1}^{\frac{4}{N-2}}\interleave W_n^l+w_n^l\interleave_{I}^{\frac{4}{N-2}-\alpha} \interleave r_n^l \interleave_{I}^{\alpha+1}+ C\interleave r_n^l\interleave_I \interleave W_n^l+w_n^l\interleave_I^{\frac{6-N}{N-2}}\left\|W_n^l+w_n^l) \right\|_{L^{\frac{2(N+2)}{N-2}}\left(I \times \mathbb{R}^N\right)}\nonumber\\
&&+C\gamma_n^l(a)+C\delta_n^l.
 \end{eqnarray}

We prepare several propositions. The first one gives a uniform bound on $\left\|W_n^l+w_n^l\right\|_{\mathbb{R}}$.

\begin{proposition}\label{P7.1}
There exists $C>0$, such that, for every $l \geqslant 1$,
 \begin{equation}\label{eq7.14}
   \limsup\limits_{n \rightarrow \infty}\interleave W_n^l+w_n^l\interleave_{\mathbb{R}} \leq  C .
 \end{equation}
\end{proposition}
In view of \eqref{eq7.13} and Proposition \ref{P7.1}, we get for every $l \geq 1$ and $n \geq  N(l)$,
 \begin{eqnarray}\label{eq7.15}
&&\interleave r_n^l\interleave_I+\sup\limits_{t \in I} \|\nabla r_n^l(t)\|_{L^2(\mathbb{R}^N)}\nonumber\\
&\leq&C\left(\sum_{\alpha=1}^{\frac{4}{N-2}} \interleave r_n^l \interleave_{I}^{\alpha+1}+ C\interleave r_n^l\interleave_I  \left\|W_n^l+w_n^l) \right\|_{L^{\frac{2(N+2)}{N-2}}\left(I \times \mathbb{R}^N\right)}+C\gamma_n^l(a)+C\delta_n^l\right).
 \end{eqnarray}

The second proposition shows that under a suitable finite partition of $\mathbb{R}$, one can absorb the linear term in $\left\|r_n^l\right\|_I$ in the right-hand side of \eqref{eq7.15}.
\begin{proposition}\label{P7.2}
For every $\varepsilon>0$, there exists an $n$-dependent finite partition of $\mathbb{R}_{+}$
 \begin{equation}\label{eq7.16}
   \mathbb{R}_{+}=\underbrace{\left[0, a_n^1\right]}_{I_n^1} \cup \underbrace{\left[a_n^1, a_n^2\right]}_{I_n^2} \cup \cdots \cup \underbrace{\left[a_n^{p-1},+\infty\right]}_{I_n^p},
 \end{equation}
such that
\begin{equation}\label{eq7.17}
  \limsup _{n \rightarrow \infty}\left\|W_n^l+w_n^l\right\|_{L^{\frac{2(N+2)}{N-2}}\left(I_n^i \times \mathbb{R}^N\right)} \leqslant \varepsilon
\end{equation}
for every $1 \leq i \leq  p$ and every $l \geq  1$.
\end{proposition}
The third proposition proves the smallness of $\delta_n^l$.

\begin{proposition}\label{P7.3}
 \begin{equation}\label{eq7.18}
   \limsup _{n \rightarrow \infty} \delta_n^l\rightarrow 0 , \ l \rightarrow \infty.
 \end{equation}
\end{proposition}
Let us first assume the validity of Propositions \ref{P7.1}, \ref{P7.2}, \ref{P7.3}, and use them to achieve the proof of Lemma \ref{L7.3}. Applying \eqref{eq7.15} on an interval $I_n^i$, provided by Proposition \ref{P7.2}, it follows that
 \begin{equation}\label{eq7.19}
   \interleave r_n^l\interleave_{I_n^i}+\sup\limits_{t \in I_n^i} \|\nabla r_n^l(t)\|_{L^2(\mathbb{R}^N)} \leq  C\left(\gamma_n^l\left(a_n^i\right)+\delta_n^l+2 \varepsilon\left\|r_n^l\right\|_{I_n^i}+\sum_{\alpha=2}^{\frac{N+2}{N-2}}\interleave r_n^l \interleave_{I_n^i}^{\alpha }\right)
 \end{equation}
for every $l \geq 1$ and $n \geq  N(l)$. If we choose $\varepsilon$ so that $C \varepsilon<\frac{1}{2}$, we obtain
 \begin{equation}\label{eq7.20}
    \interleave r_n^l\interleave_{I_n^i}+\sup\limits_{t \in I_n^i} \|\nabla r_n^l(t)\|_{L^2(\mathbb{R}^N)} \leq C\left(\gamma_n^l\left(a_n^i\right)+\delta_n^l +\sum_{\alpha=2}^{\frac{N+2}{N-2}}\interleave r_n^l \interleave_{I_n^i}^{\alpha }\right) .
 \end{equation}
For $i=1$, \eqref{eq7.20} reads
 \begin{equation}\label{eq7.21}
    \interleave r_n^l\interleave_{I_n^1}+\sup\limits_{t \in I_n^1} \|\nabla r_n^l(t)\|_{L^2(\mathbb{R}^N)} \leq C\left(\gamma_n^l\left(0\right)+\delta_n^l +\sum_{\alpha=2}^{\frac{N+2}{N-2}}\interleave r_n^l \interleave_{I_n^1}^{\alpha }\right) .
 \end{equation}
Recall that, in view of definition of nonlinear profiles and $\gamma_n^l$, we have
 \begin{equation}\label{eq7.22}
   \gamma_n^l(0) \leqslant \sum_{j=1}^l \left\|(U-V)\left(-\frac{t_n}{h_n^2}\right)\right\|_{H^1(\mathbb{R}^N)} \rightarrow  0,\ n \rightarrow \infty
 \end{equation}
for every $l \geq 1$. Now we need the following classical lemma, see \cite{KSK2001}.
\begin{lemma}\label{L7.5}
Let $M=M(t)$ be a positive continuous function on $[0, T]$, such that $M(0)=0$ and, for every $t \in[0, T]$, we have
$$
M(t) \leqslant C\left(a+\sum_{\alpha=2}^{\frac{N+2}{N-2}} M^\alpha(t)\right),
$$
with $0<a<a_0=a_0(C)$. Then we have, for every $t \in[0, T]$,
$$
M(t) \leq  2 C a .
$$
\end{lemma}
From \eqref{eq7.18} and \eqref{eq7.22}, it follows that, for every $l$ large enough, there exists $N(l)$, such that
$$
\gamma_n^l(0)+\delta_n^l \leq  a_0(C)
$$
for every $n \geq  N(l)$. We denote by $M_n^l$ the function defined on $I_n^1=\left[0, a_n^1\right]$ by
$$
M_n^l(s)= \interleave r_n^l \interleave_{[0, s]}+\frac{s}{a_n^1}\sup\limits_{t \in I_n^1} \|\nabla r_n^l(s)\|_{L^2(\mathbb{R}^N)}   .
$$
It is clear that \eqref{eq7.20}  still holds if we replace $I_n^1=\left[0, a_n^1\right]$ by $[0, s]$ for every $s \in I_n^1$. Thus,
$$
M_n^l(s) \leq C\left(\gamma_n^l(0)+\delta_n^l+\sum_{\alpha=2}^{\frac{N+2}{N-2}}\left(M_n^l\right)^\alpha(s)\right) .
$$
Hence, $M_n^l$ satisfies the conditions of Lemma \ref{L7.5} for $l$ large and $n \geqslant N(l)$, and so
 \begin{equation}\label{eq7.23}
   M_n^l\left(a_n^1\right)=\interleave r_n^l\interleave_{I_n^1}+\sup\limits_{t \in I_n^1} \|\nabla r_n^l(t)\|_{L^2(\mathbb{R}^N)}  \leq  2 C\left(\gamma_n^l(0)+\delta_n^l\right)
 \end{equation}
for $l$ large and $n \geq N(l)$. Summing \eqref{eq7.18}, \eqref{eq7.22}, and \eqref{eq7.23}, it follows that
$$
\limsup _{n \rightarrow \infty}\left(\interleave r_n^l\interleave_{I_n^1}+\sup\limits_{t \in I_n^1} \|\nabla r_n^l(t)\|_{L^2(\mathbb{R}^N)} \right)\rightarrow 0,\ l \rightarrow \infty.
$$
On the other hand, in view of notation \eqref{eq6.99}, we have
$$
\gamma_n^l\left(a_n^1\right) \leq \sup\limits_{t \in I_n^1} \|\nabla r_n^l(t)\|_{L^2(\mathbb{R}^N)}  ,
$$
which implies
$$
\limsup _{n \rightarrow \infty} \gamma_n^l\left(a_n^1\right) \rightarrow 0 ,\ l \rightarrow \infty.
$$
This fact allows us to repeat the same argument on $I_n^2=\left[a_n^1, a_n^2\right]$. We get, similarly,
$$
 \interleave r_n^l \interleave_{I_n^2}+\sup\limits_{t \in I_n^2} \|\nabla r_n^l(t)\|_{L^2(\mathbb{R}^N)}  \leq C\left(\gamma_n^l\left(a_n^1\right)+\delta_n^l\right),
$$
so
$$
\limsup _{n \rightarrow \infty}\left(\interleave r_n^l\interleave_{I_n^2}+\sup\limits_{t \in I_n^2} \|\nabla r_n^l(t)\|_{L^2(\mathbb{R}^N)}\right)\rightarrow 0,\ l \rightarrow \infty .
$$
By iterating this process, we obtain that, for every $1 \leq  i \leq p$,
$$
\limsup _{n \rightarrow \infty}\left(\interleave r_n^l\interleave_{I_n^i}+\sup\limits_{t \in I_n^i} \|\nabla r_n^l(t)\|_{L^2(\mathbb{R}^N)}\right) \rightarrow 0,\ l \rightarrow \infty .
$$
Since $p$ does not depend on $n$ and $l$, we get
$$
\limsup _{n \rightarrow \infty}\left(\interleave r_n^l\interleave_{\mathbb{R}_{+}}+\sup _{t \in \mathbb{R}_{+}}\|\nabla r_n^l(t)\|_{L^2(\mathbb{R}^N)}\right)  \rightarrow 0,\ l \rightarrow \infty .
$$
A similar decomposition to \eqref{eq7.16} for $\mathbb{R}_{-}$can be provided. Arguing as above we prove
$$
\limsup _{n \rightarrow \infty}\left(\interleave r_n^l\interleave_{\mathbb{R}_{-}}+\sup _{t \in \mathbb{R}_{-}}\|\nabla r_n^l(t)\|_{L^2(\mathbb{R}^N)}\right)  \rightarrow 0,\ l \rightarrow \infty .
$$
Hence, we get
$$
\limsup _{n \rightarrow \infty}\left(\interleave r_n^l\interleave_{\mathbb{R}}+\sup _{t \in \mathbb{R}}\|\nabla r_n^l(t)\|_{L^2(\mathbb{R}^N)}\right)  \rightarrow 0,\ l \rightarrow \infty .
$$
Since
$$
E\left(r_n^l(t)\right) \leq  \|\nabla r_n^l(t)\|_{L^2(\mathbb{R}^N)}^2\left(1+\|\nabla r_n^l(t)\|_{L^2(\mathbb{R}^N)}^{\frac{4}{N-2}}\right)\rightarrow 0,\ l \rightarrow \infty  ,
$$
it follows that
$$
\limsup _{n \rightarrow \infty}\left(\interleave r_n^l\interleave_{\mathbb{R}}+\sup _{t \in \mathbb{R}} E \left(r_n^l(t)\right)\right) \rightarrow 0,\ l \rightarrow \infty .
$$

To complete the proof of the lemma, we need to prove Propositions \ref{P7.1}, \ref{P7.2}, and \ref{P7.3}. The proof of Propositions \ref{P7.1} and \ref{P7.2} can be found in \cite{KSK2001}. Since this paper considers general nonlinear terms, we need to modify the proof of the Proposition \ref{P7.3}. The proof is divided into two parts. In the first one we prove that, for every $l \geq  1$,
 \begin{equation}\label{eq7.25}
   \left\|\nabla\left(\sum_{j=1}^l p\left(U_n^j\right)-p\left(\sum_{j=1}^l U_n^j\right)\right)\right\|_{L^{\frac{2(N+2)}{N+4}}\left(\mathbb{R}\times\mathbb{R}^N\right)}\rightarrow 0,\ n \rightarrow \infty .
 \end{equation}
In the second and main one, we prove
 \begin{equation}\label{eq7.26}
\limsup _{n \rightarrow \infty}\left\|\nabla\left(p\left(W_n^l+w_n^l\right)-p\left(W_n^l\right)\right)\right\|_{L^{\frac{2(N+2)}{N+4}}\left(\mathbb{R}\times\mathbb{R}^N\right)}\rightarrow 0,\ l \rightarrow \infty .
 \end{equation}

\textbf{Part 1}. Arguing in the same way as in the proof of Lemma 2.7 and Proposition 3.4 in
 \cite{KSK2001}, we concludes the proof of \eqref{eq7.25}.

\textbf{Part 2}. By Leibnitz formula and H\"older inequality we get
 \begin{eqnarray}\label{eq7.27}
&&\left\|\nabla\left(p\left(W_n^l+w_n^l\right)-p\left(W_n^l\right)\right)\right\|_{L^{\frac{2(N+2)}{N+4}}\left(\mathbb{R}\times\mathbb{R}^N\right)} \nonumber\\
&\leq&C\left\| |\nabla W_n^l|\left(W_n^l+w_n^l \right)^{\frac{6-N}{N-2}}|w_n^l|\right\|_{L^{\frac{2(N+2)}{N+4}}\left(I \times \mathbb{R}^N\right)}+C\left\||\nabla w_n^l|\left(W_n^l \right)^{\frac{4}{N-2}} \right\|_{L^{\frac{2(N+2)}{N+4}}\left(I \times \mathbb{R}^N\right)} \nonumber\\
&&+ \varepsilon\|\nabla  w_n^l\|_{L^{\frac{2(N+2)}{N+4}}\left(I \times \mathbb{R}^N\right)} \nonumber\\
&\leq&C\left\|  \nabla W_n^l  \right\|_{L^{\frac{2(N+2)}{N }}\left(I \times \mathbb{R}^N\right)}\left\| W_n^l+w_n^l \right\|_{L^{\frac{2(N+2)}{N-2 }}\left(I \times \mathbb{R}^N\right)}^{\frac{6-N}{N-2}}\left\|  w_n^l \right\|_{L^{\frac{2(N+2)}{N-2}}\left(I \times \mathbb{R}^N\right)} \nonumber\\
&&+C\left\||W_n^l \nabla w_n^l|\left(W_n^l \right)^{\frac{6-N}{N-2}} \right\|_{L^{\frac{2(N+2)}{N+4}}\left(I \times \mathbb{R}^N\right)}+ \varepsilon\|\nabla  w_n^l\|_{L^{\frac{2(N+2)}{N+4}}\left(I \times \mathbb{R}^N\right)} \nonumber\\
& \leq& C\left(\left\|w_n^l\right\|_{L^{\frac{2(N+2)}{N-2}}\left(\mathbb{R}\times\mathbb{R}^N\right)}\interleave W_n^l+w_n^l\interleave_{\mathbb{R}}^{\frac{4}{N-2}}+\interleave W_n^l\interleave_{\mathbb{R}}^{\frac{6-N}{N-2}} \left\|W_n^l \nabla\left(w_n^l\right)\right\|_{L^{\frac{N+2}{N-1}}\left(\mathbb{R} \times\mathbb{R}^N \right)}\right)\nonumber\\
&&+o(1) .
 \end{eqnarray}

Putting together \eqref{eq5.99}, \eqref{eq7.14}  and \eqref{eq7.27}, it follows that if we prove
$$
\limsup _{n \rightarrow \infty}\left\|W_n^l \nabla\left(w_n^l\right)\right\|_{L^{\frac{N+2}{N-1}}\left(\mathbb{R} \times\mathbb{R}^N \right)} \rightarrow  0,\ l \rightarrow \infty ,
$$
then the proof of \eqref{eq7.26} is done. On the other hand, the convergence of the the series $\sum\limits_{j \geq  1}\left\|U^j\right\|_{L^{\frac{2(N+2)}{N-2}}\left(\mathbb{R} \times\mathbb{R}^N \right)} ^{\frac{2(N+2)}{N-2}}$ implies that, for every $\varepsilon>0$, there exists $l(\varepsilon)$, such that
$$
\sum\limits_{j \geq  l(\varepsilon)}\left\|U^j\right\|_{L^{\frac{2(N+2)}{N-2}}\left(\mathbb{R} \times\mathbb{R}^N \right)} ^{\frac{2(N+2)}{N-2}} \leq  \varepsilon^{\frac{2(N+2)}{N-2}} .
$$
In particular,
 \begin{eqnarray}\label{eq7.29}
&&\limsup _{n \rightarrow \infty}\left\|\left(\sum_{j=l(\varepsilon)}^l U_n^j\right) \nabla\left(w_n^l\right)\right\|_{L^{\frac{N+2}{N-1}}\left(\mathbb{R} \times\mathbb{R}^N \right)}^{\frac{2(N+2)}{N-2}}\nonumber\\
& \leq& \left(\sum\limits_{j \geq  l(\varepsilon)}\left\|U^j\right\|_{L^{\frac{2(N+2)}{N-2}}\left(\mathbb{R} \times\mathbb{R}^N \right)} ^{\frac{2(N+2)}{N-2}} \right) \limsup _{n \rightarrow \infty}\left\|\nabla w_n^l\right\|_{L^{\frac{2(N+2)}{N}}\left(\mathbb{R} \times\mathbb{R}^N \right)} ^{\frac{2(N+2)}{N-2}} \nonumber\\
& \leq&  C \varepsilon^{\frac{2(N+2)}{N-2}} .
 \end{eqnarray}
The last inequality follows from the fact that $\left\|\nabla w_n^l\right\|_{L^{\frac{2(N+2)}{N}}\left(\mathbb{R} \times\mathbb{R}^N \right)}$ is uniformly bounded. Thereby,
$$
\limsup\limits_{n \rightarrow \infty}\left\|W_n^l \nabla\left(w_n^l\right)\right\|_{L^{\frac{N+2}{N-1}}\left(\mathbb{R} \times\mathbb{R}^N \right)} \leq  \limsup\limits_{n \rightarrow \infty}\left\|W_n^{l(\varepsilon)} \nabla\left(w_n^l\right)\right\|_{L^{\frac{N+2}{N-1}}\left(\mathbb{R} \times\mathbb{R}^N \right)}+C \varepsilon,
$$
for $l \geqslant l(\varepsilon)$. Hence, our problem is reduced to prove that
$$
\limsup _{n \rightarrow \infty}\left\|W_n^{l_0} \nabla\left(w_n^l\right)\right\|_{L^{\frac{N+2}{N-1}}\left(\mathbb{R} \times\mathbb{R}^N \right)} \rightarrow 0,\ l \rightarrow \infty
$$
for every fixed $l_0 \geq  1$. Since $W_n^{l_0}=\sum\limits_{j=1}^{l_0} U_n^j$, we have to show that
$$
\limsup _{n \rightarrow \infty}\left\|U_n^j \nabla\left(w_n^l\right)\right\|_{L^{\frac{N+2}{N-1}}\left(\mathbb{R} \times\mathbb{R}^N \right)} \rightarrow  0,\ l \rightarrow \infty
$$
for every $l_0 \geq  j \geq 1$. A change of variables $x=h_n^j y+x_n, t=\left(h_n^j\right)^2 s+t_n$ gives
$$
\left\|U_n^j \nabla\left(w_n^l\right)\right\|_{L^{\frac{N+2}{N-1}}\left(\mathbb{R} \times\mathbb{R}^N \right)}=\left\|U^j \nabla\left(\widetilde{w}_n^l\right)\right\|_{L^{\frac{N+2}{N-1}}\left(\mathbb{R} \times\mathbb{R}^N \right)},
$$
where
$$
\widetilde{w}_l^n(y, s)=(h_n^j)^{\frac{N-2}{2}} w_n^l\left(h_n^j y+x_n,\left(h_n^j\right)^2 s+t_n  \right) .
$$
Observe that
$$
\left\|w_n^l\right\|_{L^{\frac{2(N+2)}{N-2}}\left(\mathbb{R} \times\mathbb{R}^N \right)}=\left\|\widetilde{w}_n^l\right\|_{L^{\frac{2(N+2)}{N-2}}\left(\mathbb{R} \times\mathbb{R}^N \right)} \ \text { and } \ \left\|\nabla w_n^l\right\|_{L^{\frac{2(N+2)}{N}}\left(\mathbb{R} \times\mathbb{R}^N \right) }=\left\|\nabla \widetilde{w}_n^l\right\|_{L^{\frac{2(N+2)}{N}}\left(\mathbb{R} \times\mathbb{R}^N \right)} .
$$
By density, we can take $U^j \in C_0^{\infty}$. Using H\"older inequality, it is enough to prove
 \begin{equation}\label{eq7.30}
   \limsup _{n \rightarrow \infty}\left\|\nabla \widetilde{w}_n^l\right\|_{L^2(\mathscr{B})} \rightarrow 0,\ l \rightarrow \infty,
 \end{equation}
where $\mathscr{B}$ is a fixed compact of $\mathbb{R}_t \times \mathbb{R}_x^N$. This fact will follow from the following Lemma which can be founded in \cite{KSK2001}.

\begin{lemma}\label{L7.6}
Let $\mathscr{B}$ be a compact set of $\mathbb{R}_t \times \mathbb{R}_x^N$. Then, for every $\varepsilon>0$, there exists a constant $C_{\varepsilon}$ such that
$$
\|\nabla v\|_{L^2(\mathscr{B})} \leq C_{\varepsilon}\|v\|_{L^{\frac{2(N+2)}{N-2}}\left(\mathbb{R} \times\mathbb{R}^N \right)}+\varepsilon\|\nabla v(0, \cdot)\|_{L^2\left(\mathbb{R}^N\right)}
$$
for every $v(t, x)$ solution to linear Schr\"odinger equation \eqref{eq7.5}.
\end{lemma}
Applying Lemma \ref{L7.6} and \eqref{eq5.99}, we know that \eqref{eq7.30} holds. This concludes the proof of Lemma \ref{L7.3}.
\end{proof}

\begin{proof}[\bf Proof of Lemma \ref{L5.4} ]
\eqref{eq5.14} is a consequence of the proof of Corollary 1.9 in \cite{KSK2001}, here, we use the hypothesis $\left\|e^{i t \Delta} v_{0, n}\right\|_{L_{t,x}^{\frac{2(N+2)}{N-2}}} \geq \rho>0$. \eqref{eq5.17} follows from the orthogonality of $\left(\lambda_{j, n} ; x_{j, n} ; t_{j, n}\right)$ as in the proof of \eqref{eq5.16}. The rest of the lemma is contained in the proof of Lemmas \ref{L7.2} and \ref{L7.3}.
\end{proof}

\section{Compactness of critical element}

Let us consider the statement:

$(SC)$ For all $u_0 \in H^1(\mathbb{R}^N)$, with
\begin{equation*}
   \int_{\mathbb{R}^N}\left|\nabla u_0\right|^2dx<\int_{\mathbb{R}^N}|\nabla W|^2dx\ \text{and}\ E\left(u_0\right)<E(W),
\end{equation*}
if $u$ is the corresponding solution to the \eqref{eq1.1}, with maximal interval of existence $I$, then $I=(-\infty,+\infty)$ and $\|u\|_{L_{(-\infty, +\infty)}^{\frac{2(N+2)}{N-2}}W^{1,{\frac{2N(N+2)}{N^2+4}}}}<\infty$.

We say that $(S C)(u_0)$ holds if for this particular $u_0$(such as, take $\left.u\right|_{t=0}=u_0)$, with
\begin{equation*}
   \int_{\mathbb{R}^N}\left|\nabla u_0\right|^2dx<\int_{\mathbb{R}^N}|\nabla W|^2dx\ \text{and}\ E\left(u_0\right)<E(W)
\end{equation*}
and $u$ the corresponding solution to the (CP), with maximal interval of existence $I$, we have $I=(-\infty,+\infty)$ and $\|u\|_{L_{(-\infty, +\infty)}^{\frac{2(N+2)}{N-2}}W^{1,{\frac{2N(N+2)}{N^2+4}}}}<\infty$.

Note that, because of Lemma \ref{L3.2}, if $\left\|\nabla u_0\right\|_{L^2} \leq \bar{\delta},(S C)\left(u_0\right)$ holds. Thus, in light of Corollary \ref{c5.2}, there exists $\eta_0>0$ such that, if $u_0$ is as in $(\mathrm{SC})$ and $E\left(u_0\right)<\eta_0$, then $(S C)\left(u_0\right)$ holds. Moreover, for any $u_0$ as in (SC), $E\left(u_0\right) \geq 0$, in light of Theorem \ref{t5.1}. Thus, there exists a number $E_C$, with $\eta_0 \leq E_C \leq E(W)$, such that, if $u_0$ is as in (SC) and $E\left(u_0\right)<E_C$, $(S C)\left(u_0\right)$ holds and $E_C$ is optimal with this property. For the rest of this section we will assume that $E_C<E(W)$. We now prove that there exits a critical element $u_{0, C}$ at the critical level of energy $E_C$ so that $(S C)\left(u_{0, C}\right)$ does not hold and from the minimality, this element has a compactness property up to the symmetry of this equation. This is in fact a general principle which follows from the concentration compactness ideas. More precisely,
\begin{lemma}\label{L5.2}
There exists $u_{0, C}$ in $H^1$, with
$$
E(u_{0, C})=E_C<E(W), \quad \int_{\mathbb{R}^N}\left|\nabla u_{0, C}\right|^2dx<\int_{\mathbb{R}^N}|\nabla W|^2dx
$$
such that, if $u_C$ is the solution of \eqref{eq1.1} with data $u_{0, C}$, and maximal interval of existence $I, 0 \in \stackrel{\circ}{I}$, then $\left\|u_C\right\|_{L_{I}^{\frac{2(N+2)}{N-2}}W^{1,{\frac{2N(N+2)}{N^2+4}}}}=+\infty$.
\end{lemma}

\begin{lemma}\label{L5.3}
Assume $u_C$ is as in Lemmas \ref{L5.2} and  $\left\|u_C\right\|_{L_{I_{+}}^{\frac{2(N+2)}{N-2}}W^{1,{\frac{2N(N+2)}{N^2+4}}}}=+\infty$, where $I_{+}=(0,+\infty) \cap I$. Then there exists $x(t) \in \mathbb{R}^N$ and $\lambda(t) \in \mathbb{R}^{+}$, for $t \in I_{+}$, such that
$$
K=\left\{v(x, t): v(x, t)=\frac{1}{\lambda(t)^{\frac{N-2}{2}}} u_C\left(\frac{x-x(t)}{\lambda(t)}, t\right)\right\}
$$
has the property that $\overline{K}$ is compact in $H^1$. A corresponding conclusion is reached if
\begin{equation*}
  \left\|u_C\right\|_{L_{I_{-}}^{\frac{2(N+2)}{N-2}}W^{1,{\frac{2N(N+2)}{N^2+4}}}}=+\infty,
\end{equation*}
where $I_{-}=(-\infty, 0) \cap I$.
\end{lemma}

\begin{lemma}\label{L5.5}
Let $\left\{z_{0, n}\right\} \in H^1$, with
\begin{equation*}
  \int_{\mathbb{R}^N}\left|\nabla z_{0, n}\right|^2dx<\int_{\mathbb{R}^N}|\nabla W|^2dx \  \text{and}\  E\left(z_{0, n}\right) \rightarrow E_C
\end{equation*}
and with $\left\|e^{i t \Delta} z_{0, n}\right\|_{L_{(-\infty, +\infty)}^{\frac{2(N+2)}{N-2}}W^{1,{\frac{2N(N+2)}{N^2+4}}}} \geq \rho$, where $\rho$ as in Lemma \ref{L3.1}. Let $\left\{V_{0, j}\right\}$ be as in Lemma \ref{L5.4}. Assume that one of the two hypothesis
 \begin{equation}\label{eq8.1}
   \varliminf_{n \rightarrow \infty} E(V_1^l(-\frac{t_{1, n}} { \lambda_{1, n}^2}))<E_C
 \end{equation}
or after passing to a subsequence, we have that, with $s_n=-\frac{t_{1, n}} { \lambda_{1, n}^2}$, $E(V_1^l(s_n)) \rightarrow E_C$, and $s_n \rightarrow s_* \in[-\infty,+\infty]$, and if $U_1$ is the non-linear profile (see Definition \ref{D2.1} and \ref{D2.4}) associated to $\left(V_{0,1},\left\{s_n\right\}\right)$ we have that the maximal interval of existence of $U_1$ is $I=(-\infty,+\infty)$ and $\left\|U_1\right\|_{L_{(-\infty, +\infty)}^{\frac{2(N+2)}{N-2}}W^{1,{\frac{2N(N+2)}{N^2+4}}}}<\infty$ and
 \begin{equation}\label{eq8.2}
\varliminf_{n \rightarrow \infty} E(V_1^l(-\frac{t_{1, n}}{\lambda_{1, n}^2}))=E_C .
 \end{equation}
Then (after passing to a subsequence), for $n$ large, if $z_n$ is the solution of \eqref{eq1.1} with data at $t=0$ equal to $z_{0, n}$, then $(S C)\left(z_{0, n}\right)$ holds.
\end{lemma}

Let us first assume the validity of Lemma \ref{L5.5} and use it (together with Lemma \ref{L5.4}) to establish Lemmas \ref{L5.2} and \ref{L5.3}.
\begin{proof}[\bf Proof of Lemma \ref{L5.2}]
According to the definition of $E_C$, and the assumption that $E_C<E(W)$, we can find $u_{0, n} \in H^1$, with
\begin{equation*}
  \int_{\mathbb{R}^N}\left|\nabla u_{0, n}\right|^2dx<\int_{\mathbb{R}^N}|\nabla W|^2dx,\ E\left(u_{0, n}\right) \rightarrow E_C
\end{equation*}
and such that if $u_n$ is the solution of \eqref{eq1.1} with data at $t=0$, $u_{0, n}$ and maximal interval of existence $I_n=\left(-T_{-}\left(u_{0, n}\right), T_{+}\left(u_{0, n}\right)\right)$, then $\left\|e^{i t \Delta} u_{0, n}\right\|_{L_{(-\infty, +\infty)}^{\frac{2(N+2)}{N-2}}W^{1,{\frac{2N(N+2)}{N^2+4}}}} \geq \rho>0$, where $\rho$ is as in Lemma \ref{L3.1} and $\left\|u_n\right\|_{L_{I_n}^{\frac{2(N+2)}{N-2}}W^{1,{\frac{2N(N+2)}{N^2+4}}}}=+\infty$(Here we are also using Lemma \ref{L2.1} and Lemma \ref{L3.1}). Note that, since $E_C<E(W)$, there exists $\delta_0>0$ such that
\begin{equation*}
  E\left(u_{0, n}\right) \leq\left(1-\delta_0\right) E(W),\ \forall n.
\end{equation*}
Because of Theorem \ref{t5.1}, we can find $\bar{\delta}$ so that
\begin{equation*}
   \int_{\mathbb{R}^N}\left|\nabla u_n(t)\right|^2dx \leq(1-\bar{\delta}) \int_{\mathbb{R}^N}|\nabla W|^2dx \ \text{for all}\ t \in I_n, \ \forall n.
\end{equation*}
Apply now Lemma \ref{L5.4} for $\varepsilon_0>0$ and Lemma \ref{L5.5}. We then have, for $J=J\left(\varepsilon_0\right)$, that
\begin{equation}\label{eq5.18}
  u_{0, n}   =\sum_{j=1}^J \frac{1}{\lambda_{j, n}^{\frac{N-2}{2}}} V_j^l\left(\frac{x-x_{j, n}}{\lambda_{j, n}}, \frac{-t_{j, n}}{\lambda_{j, n}^2}\right)+w_n,
\end{equation}
\begin{equation}\label{eq5.19}
  \int_{\mathbb{R}^N}\left|\nabla u_{0, n}\right|^2dx   =\sum_{j=1}^J \int_{\mathbb{R}^N}\left|\nabla V_{0, j}\right|^2dx+\int_{\mathbb{R}^N}\left|\nabla w_n\right|^2dx+o(1),
\end{equation}
\begin{equation}\label{eq5.20}
  E\left(u_{0, n}\right)  =\sum_{j=1}^J E\left(V_j^l\left(\frac{-t_{j, n}}{\lambda_{j, n}^2}\right)\right)+E\left(w_n\right)+o(1) .
\end{equation}
Note that because of \eqref{eq5.19} we have, for all $n$ large, that
\begin{equation*}
  \int_{\mathbb{R}^N}\left|\nabla w_n\right|^2dx \leq (1-\frac{\bar{\delta}}{2}) \int_{\mathbb{R}^N}|\nabla W|^2dx \ \text{and}\  \int_{\mathbb{R}^N}\left|\nabla V_{0, j}\right|^2dx \leq(1-\frac{\bar{\delta}}{2}) \int_{\mathbb{R}^N}|\nabla W|^2dx.
\end{equation*}
From Corollary \ref{c5.1} it now follows that $E (V_j^l (-\frac{t_{j, n}}{ \lambda_{j n}^2} ) ) \geq 0$ and $E\left(w_n\right) \geq 0$. From this and \eqref{eq5.20} it follows that
\begin{equation*}
  E(V_1^l(-\frac{t_{1, n}}{\lambda_{1, n}^2})) \leq E\left(u_{0, n}\right)+o(1)
\end{equation*}
and hence $\varliminf\limits_{n \rightarrow \infty} E(V_1^l(-\frac{t_{1, n}}{\lambda_{1, n}^2})) \leq E_C$. If the left-hand side is strictly less than $E_C$, Lemma \ref{L5.5} gives us a contradiction with the choice of $u_{0, n}$, for $n$ large (after passing to a subsequence). Hence, the left-hand side must equal $E_C$.

Let then $U_1$ be the non-linear profile associated to $\left(V_1^l,\left\{s_n\right\}\right)$, with $s_n=-\frac{t_{1, n}}{\lambda_{1, n}^2}$(after passing to a subsequence). We first note that we must have $J=1$. This is because \eqref{eq5.20} and $E\left(u_{0, n}\right) \rightarrow E_C, E\left(V_1^l\left(-s_n\right)\right) \rightarrow E_C$ now imply that
\begin{equation*}
  E\left(w_n\right) \rightarrow 0\ \text{and}\ E(V_j^l(-\frac{t_{j, n}}{\lambda_{j, n}^2})) \rightarrow 0, j=2, \ldots, J.
\end{equation*}
Using \eqref{eq5.9} and the argument in the proof of Corollary \ref{c5.2}, we have
\begin{equation*}
  \sum_{j=2}^J \int_{\mathbb{R}^N}|\nabla V_j^l(-\frac{t_{j, n}}{\lambda_{j, n}^2})|^2dx+\int_{\mathbb{R}^N}\left|\nabla w_n\right|^2dx \rightarrow 0.
\end{equation*}
We then have, since $\int_{\mathbb{R}^N}|\nabla V_j^l(-\frac{t_{j, n}}{\lambda_{j, n}^2})|^2 dx=\int_{\mathbb{R}^N}\left|\nabla V_{0, j}\right|^2dx$ that $V_{0, j}=0, j=2, \ldots, J$ and $\int_{\mathbb{R}^N}\left|\nabla w_n\right|^2dx \rightarrow 0$. Hence \eqref{eq5.18} becomes $u_{0, n}=\frac{1}{\lambda_{1, n}^{\frac{N-2}{2}}} V_1^l\left(\frac{x-x_{1, n}}{\lambda_{1, n}}, s_n\right)+w_n$. Let $v_{0, n}=\lambda_{1, n}^{\frac{N-2}{2}} u_{0, n}\left(\lambda_{1, n}\left(x+x_{1, n}\right)\right)$ and note that scaling gives us that $v_{0, n}$ verifies the same hypothesis as $u_{0, n}$. Moreover, $\widetilde{w}_n=\lambda_{1, n}^{\frac{N-2}{2}} w_n\left(\lambda_{1, n}(x+x_{1, n})\right)$ still verifies $\int_{\mathbb{R}^N}\left|\nabla \widetilde{w}_n\right|^2dx \rightarrow 0$. Thus
$$
v_{0, n}=V_1^l(s_n)+\widetilde{w}_n, \ \int_{\mathbb{R}^N}\left|\nabla \widetilde{w}_n\right|^2dx \rightarrow 0 .
$$

Let us return to $U_1$, the non-linear profile associated to $\left(V_{0,1},\left\{s_n\right\}\right)$ and let
\begin{equation*}
  I_1=(T_{-}(U_1), T_{+}(U_1))
\end{equation*}
be its maximal interval of existence. Note that, by definition of non-linear profile, we have
\begin{equation*}
  \int_{\mathbb{R}^N}\left|\nabla U_1\left(s_n\right)\right|^2dx=\int_{\mathbb{R}^N}\left|\nabla V_1^l\left(s_n\right)\right|^2dx+o(1)\ \text{and}\ E\left(U_1\left(s_n\right)\right)=E\left(V_1^l\left(s_n\right)\right)+o(1)
\end{equation*}
Note that in this case $E(V_1^l(s_n))=E_C+o(1)$ and
\begin{equation*}
  \int_{\mathbb{R}^N}\left|\nabla V_1^l\left(s_n\right)\right|^2dx= \int_{\mathbb{R}^N}\left|\nabla V_{0,1}\right|^2dx=\int_{\mathbb{R}^N}\left|\nabla u_{0, n}\right|^2dx+o(1)<\int_{\mathbb{R}^N}|\nabla W|^2dx
\end{equation*}
for $n$ large by Theorem \ref{t5.1}. Let's fix $\bar{s} \in I_1$. Then $E\left(U_1\left(s_n\right)\right)=E\left(U_1(\bar{s})\right)$, so that
$$
E\left(U_1(\bar{s})\right)=E_C .
$$
Moreover, $\int_{\mathbb{R}^N}\left|\nabla U_1\left(s_n\right)\right|^2dx<\int_{\mathbb{R}^N}|\nabla W|^2dx$ for $n$ large and hence by \eqref{eq5.11}
\begin{equation*}
  \int_{\mathbb{R}^N}\left|\nabla U_1(\bar{s})\right|^2dx<\int_{\mathbb{R}^N}|\nabla W|^2dx.
\end{equation*}
If $\left\|U_1\right\|_{L_{I_1}^{\frac{2(N+2)}{N-2}}W^{1,{\frac{2N(N+2)}{N^2+4}}}}<+\infty$, Remark  \ref{r1.4} gives us that $I_1=(-\infty,+\infty)$ and we then obtain a contradiction from Lemma \ref{L5.5}. Thus,
$$
\left\|U_1\right\|_{L_{I_1}^{\frac{2(N+2)}{N-2}}W^{1,{\frac{2N(N+2)}{N^2+4}}}}=+\infty
$$
and we then set $u_C=U_1$$($after a translation in time to make $\bar{s}=0)$.
\end{proof}

\begin{proof}[\bf Proof of Lemma \ref{L5.3}]
We argue by contradiction. For brevity of notation, let us set $u(x, t)=u_C(x, t)$. If not, there exists $\eta_0>0$ and a sequence $\left\{t_n\right\}_{n=1}^{\infty}$, $t_n \geq 0$ such that, for all $\lambda_0 \in \mathbb{R}^{+}, x_0 \in \mathbb{R}^N$, we have
\begin{equation}\label{eq5.21}
  \left\|\frac{1}{\lambda_0^{\frac{N-2}{ 2}}} u\left(\frac{x-x_0}{\lambda_0}, t_n\right)-u\left(x, t_{n^{\prime}}\right)\right\|_{H^1} \geq \eta_0, \quad \text { for } n \neq n^{\prime} .
\end{equation}
Note that $($after passing to a subsequence, so that $t_n \rightarrow \bar{t} \in\left[0, T_{+}\left(u_0\right)\right])$, we must have $\bar{t}=T_{+}\left(u_0\right)$, in view of the continuity of the flow in $H^1$, as guaranteed by Lemma \ref{L3.1}. Note that, in view of Lemma \ref{L3.1} we must also have $\left\|e^{i t \Delta} u\left(t_n\right)\right\|_{L_{(0, +\infty)}^{\frac{2(N+2)}{N-2}}W^{1,{\frac{2N(N+2)}{N^2+4}}}} \geq \rho$.

\textbf{Step 1.} Let us apply Lemma \ref{L5.4} to $v_{0, n}=u\left(t_n\right)$ with $\varepsilon_0>0$. We will show that $J=1$. Indeed, if $\varliminf\limits_{n \rightarrow \infty} E(V_1^l(-\frac{t_{1, n}}{ \lambda_{1, n}^2}))<E_C$, then by Theorem \ref{t5.1}, we have
\begin{equation*}
  \int_{\mathbb{R}^N}|\nabla u(t)|^2dx \leq (1-\delta_1) \int_{\mathbb{R}^N}|\nabla W|^2dx \ \text{for all}\  t \in I_{+}
\end{equation*}
and $E(u(t))=E\left(u_0\right)=$ $E_C<E(W)$, by Lemma \ref{L5.5} we obtain that $(SC)(u)$ holds. So $\left\|u\right\|_{L_{I_{+}}^{\frac{2(N+2)}{N-2}}W^{1,{\frac{2N(N+2)}{N^2+4}}}}<+\infty$, which contradicts the hypothesis. Therefore, it follows that $\varliminf\limits_{n \rightarrow \infty} E(V_1^l(-\frac{t_{1, n}}{\lambda_{1, n}^2}))=E_C$. Similar to the proof of Lemma \ref{L5.2}, we get \begin{equation*}
  J=1,\ \int_{\mathbb{R}^N}\left|\nabla w_n\right|^2dx \rightarrow 0.
\end{equation*}
Thus, we have
\begin{equation}\label{eq5.22}
  u(t_n)=\frac{1}{\lambda_{1, n}^{\frac{N-2}{ 2}}} V_1^l\left(\frac{x-x_{1, n}}{\lambda_{1, n}}, \frac{-t_{1, n}}{\lambda_{1, n}^2}\right)+w_n, \quad \int_{\mathbb{R}^N}\left|\nabla w_n\right|^2dx \rightarrow 0 .
\end{equation}

\textbf{Step 2.} we prove that $s_n=\frac{-t_{1, n}}{\lambda_{1, n}^2}$ must be bounded. In fact, note that
$$
e^{i t \Delta} u\left(t_n\right)=\lambda_{1, n}^{-\frac{N-2}{2}} V_1^l\left(\frac{x-x_{1, n}}{\lambda_{1, n}}, \frac{t-t_{1, n}}{\left(\lambda_{1, n}\right)^2}\right)+e^{i t \Delta} w_n .
$$

On the other hand, assume $\frac{t_{1, n}}{\lambda_{1, n}^2 }\leq-C_0$, where $C_0$ is a large positive constant. Then, since
\begin{equation*}
  \left\|e^{i t \Delta} w_n\right\|_{L_{(0, +\infty)}^{\frac{2(N+2)}{N-2}}W^{1,{\frac{2N(N+2)}{N^2+4}}}}<\frac{\rho}{ 2} \ \text{for}\ n \ \text{large enough}
\end{equation*}
and
$$
\left\|\lambda_{1, n}^{-\frac{N-2}{ 2}} V_1^l\left(\frac{x-x_{1, n}}{\lambda_{1, n}}, \frac{t-t_{1, n}}{\left(\lambda_{1, n}\right)^2}\right)\right\|_{L_{(0, +\infty)}^{\frac {2(N+2)}{N-2}}W^{1,{\frac{2N(N+2)}{N^2+4}}}} \leq\left\|V_1^l(y, s)\right\|_{L_{(C_0, +\infty)}^{\frac{2(N+2)}{N-2}}W^{1,{\frac{2N(N+2)}{N^2+4}}}} \leq \frac{\rho}{2}
$$
for $C_0$ large, which contradicts  $\left\|e^{i t \Delta} u\left(t_n\right)\right\|_{L_{(0, +\infty)}^{\frac{2(N+2)}{N-2}}W^{1,{\frac{2N(N+2)}{N^2+4}}}} \geq \rho$.

On the other hand, assume that $\frac{t_{1, n}}{\lambda_{1, n}^2} \geq C_0$, for a large positive constant $C_0, n$ large, we have
\begin{equation*}
  \left\|\lambda_{1, n}^{-\frac{N-2}{2}} V_1^l\left(\frac{x-x_{1, n}}{\lambda_{1, n}}, \frac{t-t_{1, n}}{\left(\lambda_{1, n}\right)^2}\right)\right\|_{L_{(-\infty,0)}^{\frac{2(N+2)}{N-2}}W^{1,{\frac{2N(N+2)}{N^2+4}}}}   \leq\left\|V_1^l(y, s)\right\|_{L_{(-\infty,-C_0)}^{\frac{2(N+2)}{N-2}}W^{1,{\frac{2N(N+2)}{N^2+4}}}} \leq \frac{\rho}{2}
\end{equation*}
for $C_0$ large. Hence, $\left\|e^{i t \Delta} u\left(t_n\right)\right\|_{L_{(-\infty,0)}^{\frac{2(N+2)}{N-2}}W^{1,{\frac{2N(N+2)}{N^2+4}}}} \leq \rho$, for $n$ large. By Lemma \ref{L3.1}, we know that $\|u\|_{L_{(-\infty,t_n)}^{\frac{2(N+2)}{N-2}}W^{1,{\frac{2N(N+2)}{N^2+4}}}} \leq \rho$, which gives us a contradiction because of $t_n \rightarrow T_{+}\left(u_0\right)$. Thus $\left|\frac{t_{1, n}}{ \lambda_{1, n}^2}\right| \leq C_0$ and after passing to a subsequence,
$$
\frac{t_{1, n}}{\lambda_{1, n}^2} \rightarrow t_0 \in(-\infty,+\infty) .
$$

\textbf{Step 3.} By \eqref{eq5.21} and \eqref{eq5.22}, for $n \neq n^{\prime}$ large (independently of $\left.\lambda_0, x_0\right)$, it holds
\begin{equation*}
  \left\| \frac{1}{\lambda_0^{\frac{N-2}{2}}} \frac{1}{\lambda_{1, n}^{\frac{N-2}{ 2}}} V_1^l\left(\frac{\frac{x-x_0}{\lambda_0}-x_{1, n}}{\lambda_{1, n}},\frac{-t_{1, n}}{(\lambda_{1, n})^2}\right)-\frac{1}{\left(\lambda_{1, n^{\prime}}\right)^{\frac{N-2}{ 2}}} V_1^l\left(\frac{x-x_{1, n^{\prime}}}{\lambda_{1, n^{\prime}}},\frac{-t_{1, n^{\prime}}}{(\lambda_{1, n^{\prime}})^2}\right) \right\|_{H^1} \geq \eta_0 / 2
\end{equation*}
or
\begin{equation*}
  \left\|\left(\frac{\lambda_{1, n^{\prime}}}{\lambda_{1, n} \lambda_0}\right)^{\frac{N-2}{ 2}} V_1^l\left(\frac{y \lambda_{1, n^{\prime}}}{\lambda_0 \lambda_{1, n}}+\widetilde{x}_{n, n^{\prime}}-\widetilde{x}_0,-\frac{t_{1, n}}{\left(\lambda_{1, n}\right)^2}\right)
  -V_1^l\left(y,-\frac{t_{1, n^{\prime}}}{\lambda_{1, n^{\prime}}^2}\right) \right\|_{H^1} \geq \frac{\eta_0}{2},
\end{equation*}
where $\widetilde{x}_{n, n^{\prime}}$ is a suitable point in $\mathbb{R}^N$ and $\lambda_0, \widetilde{x}_0$ are arbitrary. But if we choose $\lambda_0=\frac{\lambda_{1, n^{\prime}}}{ \lambda_{1, n}}, \tilde{x}_0 =x_{n, n^{\prime}}$, then $-\frac{t_{1, n}}{\left(\lambda_{1, n}\right)^2} \rightarrow -t_0$ and $\frac{-t_{1, n^{\prime}}}{\left(\lambda_{1, n^{\prime}}\right)^2} \rightarrow-t_0$. So $  \left\|0\right\|_{H^1} \geq \frac{\eta_0}{2}$, which reaches a contradiction.
\end{proof}

Thus, to complete the proofs of Lemmas \ref{L5.2} and \ref{L5.3} we only need to provide the proof of Lemma \ref{L5.5}.

\begin{proof}[\bf Proof of Lemma \ref{L5.5}]
Let us assume first that \eqref{eq8.2} holds and set
\begin{equation*}
  A= \int_{\mathbb{R}^N}|\nabla W|^2dx, A^{\prime}=\int_{\mathbb{R}^N}|\nabla W|^2dx, M=\left\|U_1\right\|_{L_{(-\infty, +\infty)}^{\frac{2(N+2)}{N-2}}W^{1,{\frac{2N(N+2)}{N^2+4}}}}.
\end{equation*}
Arguing (for some $\varepsilon_0>0$ in Lemma \ref{L5.4}) as in the proof of Lemmas \ref{L5.2}, we see that
\begin{equation*}
  \varliminf\limits_{n \rightarrow \infty} E(V_1^l(-\frac{t_{1, n}}{ \lambda_{1, n}^2})) =E_C \text{ and } E_C<E(W),
\end{equation*}
which imply that $J=1$, $\int_{\mathbb{R}^N}\left|\nabla w_n\right|^2dx \rightarrow 0$. Moreover, if
 \begin{equation*}
v_{0, n} =\lambda_{1, n}^{\frac{N-2}{2}} z_{0, n}(\lambda_{1, n}(x+x_{1, n})),\ \widetilde{w}_n=\lambda_{1, n}^{\frac{N-2}{2}} w_n(\lambda_{1, n}(x+x_{1, n})),\ s_n=-\frac{t_{1, n}}{\lambda_{1, n}^2},
 \end{equation*}
we have $\int_{\mathbb{R}^N}\left|\nabla \widetilde{w}_n\right|^2dx \rightarrow 0$ and $v_{0, n}=V_1^l\left(s_n\right)+\widetilde{w}_n$, while
\begin{equation*}
  \left\|e^{i t \Delta} v_{0, n}\right\|_{L_{(-\infty, +\infty)}^{\frac{2(N+2)}{N-2}}W^{1,{\frac{2N(N+2)}{N^2+4}}}} \geq \delta,\ \int_{\mathbb{R}^N}\left|\nabla v_{0, n}\right|^2dx<\int_{\mathbb{R}^N}|\nabla W|^2dx,\ E\left(v_{0, n}\right) \rightarrow E_C.
\end{equation*}
By definition of non-linear profile, we know that
\begin{equation*}
  \int_{\mathbb{R}^N}|\nabla  V_1^l(s_n)- \nabla U_1(s_n)|^2dx=o(1).
\end{equation*}
We then have
$$
v_{0, n}=U_1(s_n)+\widetilde{\widetilde{w}}_n, \int_{\mathbb{R}^N} |\nabla \widetilde{\widetilde{w}}_n |^2dx \rightarrow 0 .
$$
Moreover, as in the proof of Lemma \ref{L5.2}, $E(U_1(0))=E_C$ and $\int_{\mathbb{R}^N}\left|\nabla U_1(t)\right|^2dx<\int_{\mathbb{R}^N}|\nabla W|^2dx$ for all $t$. We now apply Proposition \ref{P5.1}, with $\varepsilon_0<\varepsilon_0\left(M, A, A^{\prime}, N\right)$ and $n$ large, with $\tilde{u}=U_1, e \equiv 0, t_0=0, u_0=v_{0, n}$. This case now follows.

Next, assume that \eqref{eq8.1} holds, the proof is divided into five steps.

\textbf{Step 1 } We prove that for $j \geq 2$, we also have $\varliminf\limits_{n \rightarrow \infty} E(V_j^l(-\frac{t_{j, n}}{\lambda_{j, n}^2}))<E_C$. In fact, up to a subsequence, assume $\lim\limits_{n \rightarrow \infty} E(V_1^l(-\frac{t_{1, n}}{\lambda_{1, n}}))<E_C$. Due to \eqref{eq5.16}, it holds
$$
\int_{\mathbb{R}^N} |\nabla z_{0, n} |^2dx \geq \sum\limits_{j=1}^J \int_{\mathbb{R}^N} |\nabla V_{0, j} |^2dx+o(1)
$$
and since $E_C<E(W)$, for $n$ large we have $E(z_{0, n}) \leq(1-\delta_0) E(W)$, by Lemma \ref{L5.1},
\begin{equation*}
  \int_{\mathbb{R}^N}\left|\nabla z_{0, n}\right|^2dx \leq(1- \delta_1) \int_{\mathbb{R}^N}|\nabla W|^2dx \text{ and } \int_{\mathbb{R}^N}|\nabla V_{0, j}|^2dx \leq (1-\delta_1) \int_{\mathbb{R}^N}|\nabla W|^2dx.
\end{equation*}
Similarly, $\int_{\mathbb{R}^N}|\nabla w_n|^2dx \leq(1-\delta_1) \int_{\mathbb{R}^N}|\nabla W|^2dx$. By Corollary \ref{c5.1}, we have
\begin{equation*}
  E(V_j^l(-\frac{t_{j, n}}{\lambda_{j, n}^2})) \geq 0,\ E(w_n) \geq 0.
\end{equation*}
Moreover, using \eqref{eq5.14} and the proof of Corollary \ref{c5.2}, we have
\begin{equation*}
  E(V_1^l(-\frac{t_{1, n}}{\lambda_{1, n}^2})) \geq C \int_{\mathbb{R}^N}\left|\nabla V_{0,1}\right|^2dx \geq c \alpha_0=\overline{\alpha_0}>0 \text{ for } n  \text{ large}.
\end{equation*}
By \eqref{eq5.17}, it holds
$$
E(z_{0, n}) \geq \overline{\alpha_0}+\sum_{j=2}^J E(V_j^l(-\frac{t_{j, n}}{ \lambda_{j, n}^2}))+o(1) \text{ for } n  \text{ large},
$$
so the claim follows from $E\left(z_{0, n}\right) \rightarrow E_C$.

\textbf{Step 2 } We show that (after passing to a subsequence so that, for each $j$, $\lim\limits_n E(V_j^l(-\frac{t_{j, n}}{\lambda_{j, n}^2}))$ exists and $\lim\limits_n(-\frac{t_{j, n}}{\lambda_{j, n}^2})=\overline{s_j} \in[-\infty,+\infty]$ exists) if $U_j$ is the non-linear profile associated to $(V_j^l,\left\{-\frac{t_{j, n}}{\lambda_{j, n}^2}\right\})$, then $U_j$ satisfies (SC). Indeed, according to the definition of non-linear profile and Step 1, it follows that $E(U_j)<E_C$ because of $\varliminf\limits_{n \rightarrow \infty} E(V_j^l(-\frac{t_{j, n}}{\lambda_{j, n}^2}))<E_C$. Moreover, since
\begin{equation*}
  \int_{\mathbb{R}^N}|\nabla V_j^l(-\frac{t_{j, n}}{\lambda_{j, n}^2})|^2dx \leq (1-\delta_1)  \int_{\mathbb{R}^N}|\nabla W|^2dx,
\end{equation*}
the definition of non-linear profile and Theorem \ref{t5.1}, if $\bar{t} \in I_j($the maximal interval for $U_j)$, we have $\int_{\mathbb{R}^N}|\nabla U_j(\bar{t})|^2dx<\int_{\mathbb{R}^N}|\nabla W|^2dx$. By the definition of $E_C$, our claim follows. Note that the argument in the proof of Proposition \ref{P5.1} also gives that $\left\| U_j\right\|_{L_{(-\infty,+\infty)}^\frac{2(N+2)}{N-2}W^{1,{\frac{2N(N+2)}{N^2+4}}}}<+\infty$.

\textbf{Step 3 } We claim that there exists $j_0$ so that, for $j \geq j_0$ we have
 \begin{equation}\label{eq8.8}
\left\|U_j\right\|_{L_{(-\infty,+\infty)}^\frac{2(N+2)}{N-2}W^{1,{\frac{2N(N+2)}{N^2+4}}}}^{\frac{2(N+2)} {N-2}} \leq C\| V_{0, j}\|_{H^1}^{\frac{2(N+2)}{ N-2}} .
 \end{equation}
In fact, from \eqref{eq5.16}, for fixed $J$ we see that (choosing $n$ large)
 \begin{equation*}
   \sum_{j=1}^J \int_{\mathbb{R}^N}\left|\nabla V_{0, j}\right|^2dx \leq \int_{\mathbb{R}^N}\left|\nabla z_{0, n}\right|^2dx+o(1) \leq 2 \int_{\mathbb{R}^N}|\nabla W|^2dx.
 \end{equation*}
Thus, for $j \geq j_0$, we have
\begin{equation*}
  \int_{\mathbb{R}^N}\left|\nabla V_{0, j}\right|^2dx \leq \widetilde{\delta},
\end{equation*}
where $\tilde{\delta}$ is so small that $\left\|e^{i t \Delta} V_{0, j}\right\|_{L_{(-\infty,+\infty)}^\frac{2(N+2)}{N-2}W^{1,{\frac{2N(N+2)}{N^2+4}}}} \leq \rho$, with $\rho$ as in Lemma \ref{L3.1}. From the definition of non-linear profile, it then follows that $\left\|U_j\right\|_{L_{(-\infty,+\infty)}^\frac{2(N+2)}{N-2}W^{1,{\frac{2N(N+2)}{N^2+4}}}} \leq 2 \rho$, and using the integral equation
\begin{equation*}
  u(t)=e^{i t \Delta} u_0+i\int_0^t e^{i\left(t-t^{\prime}\right) \Delta} f(u) d t^{\prime}.
\end{equation*}
So $\left\|U_j(0)\right\|_{H^1} \leq C\left\|V_{0, j}\right\|_{H^1}$ and $\left\| U_j\right\|_{L_{(-\infty,+\infty)}^\frac{2(N+2)}{N-2}W^{1,{\frac{2N(N+2)}{N^2+4}}}} \leq C\left\|V_{0, j}\right\|_{H^1}$, which implies \eqref{eq8.8}.

\textbf{Step 4 } For $\varepsilon_0>0$, to be chosen, define now
$$
H_{n, \varepsilon_0}=\sum\limits_{j=1}^{J\left(\varepsilon_0\right)} \frac{1}{\lambda_{j, n}^{\frac{N-2}{ 2}}} U_j\left(\frac{x-x_{j, n}}{\lambda_{j, n}}, \frac{t-t_{j, n}}{\lambda_{j, n}^2}\right),
$$
then it follows that
\begin{equation}\label{eq8.9}
  \left\|H_{n, \varepsilon_0}\right\|_{L_{(-\infty,+\infty)}^\frac{2(N+2)}{N-2}L^{\frac{2N(N+2)}{N-2}}} \leq C_0,
\end{equation}
uniformly in $\varepsilon_0$, for $n \geq n(\varepsilon_0)$. In fact,
\begin{eqnarray*}
&&\left\|H_{n, \varepsilon_0}\right\|_{L_{(-\infty,+\infty)}^\frac{2(N+2)}{N-2}L^{\frac{2N(N+2)}{N-2}}}\\
&= & \iint\left[\sum_{j=1}^{J\left(\varepsilon_0\right)} \frac{1}{\lambda_{j, n}^{\frac{N-2}{2}}} U_j\left(\frac{x-x_{j, n}}{\lambda_{j, n}}, \frac{t-t_{j, n}}{\lambda_{j, n}^2}\right)\right]^{\frac{2(N+2)}{N-2}} \\
&\leq & C_{J\left(\varepsilon_0\right)} \sum_{j^{\prime} \neq j} \iint\left|\frac{1}{\lambda_{j, n}^{\frac{N-2}{2}}} U_j\left(\frac{x-x_{j, n}}{\lambda_{j, n}}, \frac{t-t_{j, n}}{\lambda_{j, n}^2}\right)\right|\cdot\left|\frac{1}{\lambda_{j^{\prime}, n}^{\frac{N-2}{2}}} U_{j^{\prime}}\left(\frac{x-x_{j^{\prime}, n}}{\lambda_{j^{\prime}, n}}, \frac{t-t_{j^{\prime}, n}}{\lambda_{j^{\prime}, n}^2}\right)\right|^{\frac{N+6}{N-2}}\\
&& +\sum_{j=1}^{J\left(\varepsilon_0\right)} \iint\left|\frac{1}{\lambda_{j, n}^{\frac{N-2}{2}}} U_j\left(\frac{x-x_{j, n}}{\lambda_{j, n}}, \frac{t-t_{j, n}}{\lambda_{j, n}^2}\right)\right|^{\frac{2(N+2)}{N-2}} \\
&=&\mathrm{I}+\mathrm{II} .
\end{eqnarray*}
By the orthogonality of $(\lambda_{j, n} ; x_{j, n} ; t_{j, n})$, we know that $\mathrm{II} \rightarrow 0$ for $n$ large(see Keraani \cite{KSK2001}). Hence, for $n$ large we have $\mathrm{II} \leq \mathrm{I}$. Since \eqref{eq5.16}, it follows that
\begin{eqnarray*}
\mathrm{I} & \leq& \sum_{j=1}^{j_0}\left\|U_j\right\|_{L_{(-\infty,+\infty)}^\frac{2(N+2)}{N-2}L^{\frac{2N(N+2)}{N-2}}}^{\frac{2(N+2)}{ N-2}}+\sum_{j=j_0}^{J\left(\varepsilon_0\right)}\left\|U_j\right\|_{L_{(-\infty,+\infty)}^\frac{2(N+2)}{N-2}L^{\frac{2N(N+2)}{N-2}}} ^{\frac{2(N+2)}{ N-2}}\\
& \leq& \sum_{j=1}^{j_0}\left\|U_j\right\|_{L_{(-\infty,+\infty)}^\frac{2(N+2)}{N-2}L^{\frac{2N(N+2)}{N-2}}}^{\frac{2(N+2)}{ N-2}}+C \sum_{j=j_0}^{J\left(\varepsilon_0\right)}\| V_{0, j}\|_{H^1}^{\frac{2(N+2)}{ N-2}}  \\ &\leq& \frac{C_0}{2},
\end{eqnarray*}
where $j_0$ is defined as in \eqref{eq8.8}. For $\varepsilon_0>0$, to be chosen, define
\begin{equation*}
  R_{n, \varepsilon_0}=\left|H_{n, \varepsilon_0}\right|^{\frac{4}{N-2}} H_{n, \varepsilon_0}-  \sum_{j=1}^{J\left(\varepsilon_0\right)}\left|\frac{1}{\lambda_{j, n}^{\frac{N-2}{2}}} U_j\left(\frac{x-x_{j, n}}{\lambda_{j, n}}, \frac{t-t_{j, n}}{\lambda_{j, n}^2}\right)\right|^{\frac{4}{N-2}}  \frac{1}{\lambda_{j, n}^{\frac{N-2}{2}}} U_j\left(\frac{x-x_{j, n}}{\lambda_{j, n}}, \frac{t-t_{j, n}}{\lambda_{j, n}^2}\right).
\end{equation*}
using the arguments of Keraani  \cite{KSK2001}, we get
$$
\text { For } n=n\left(\varepsilon_0\right) \text { large, }\left\|\nabla R_{n, \varepsilon_0}\right\|_{L_t^2 L_x^{\frac{2 N}{N+2}}} \rightarrow 0 \quad \text { as } n \rightarrow \infty \text {. }
$$

\textbf{Step 5 }  Finally, we apply Proposition \ref{P5.1} to obtain our purpose. Let
\begin{equation*}
  \widetilde{u}=H_{n, \varepsilon_0},\ e=R_{n, \varepsilon_0},
\end{equation*}
where $\varepsilon_0$ is still to be determined. Recall that
\begin{equation*}
  z_{0, n}=\sum_{j=1}^{J\left(\varepsilon_0\right)} \frac{1}{\lambda_{j, n}^{\frac{N-2}{2}}} V_j^l\left(\frac{x-x_{j, n}}{\lambda_{j, n}}, \frac{-t_{j, n}}{\lambda_{j, n}^2}\right) +w_n,
\end{equation*}
where $\left\|e^{i t \Delta} w_n\right\|_{L_{(-\infty,+\infty)}^\frac{2(N+2)}{N-2}L^{\frac{2N(N+2)}{N-2}}} \leq \varepsilon_0$. By the definition of non-linear profile, we now have
$$
z_{0, n}(x)=H_{n, \varepsilon_0}(x, 0)+\widetilde{w}_n(x),
$$
where, for $n$ large $\left\|e^{i t \Delta} \widetilde{w}_n\right\|_{L_{(-\infty,+\infty)}^\frac{2(N+2)}{N-2}L^{\frac{2N(N+2)}{N-2}}} \leq 2 \varepsilon_0$. Moreover, according to the orthogonality of $\left(\lambda_{j, n} ; x_{j, n} ; t_{j, n}\right)$ and Corollary \ref{c5.2}, for $n=n\left(\varepsilon_0\right)$ large, it holds
\begin{equation*}
  \int_{\mathbb{R}^N}\left|\nabla H_{n, \varepsilon_0}(t)\right|^2dx \leq 2 \sum_{j=1}^{J\left(\varepsilon_0\right)} \int_{\mathbb{R}^N}|\nabla U_j(\frac{t-t_{j, n}}{\lambda_{j, n}^2})|^2dx \leq 4 C \sum_{j=1}^{J\left(\varepsilon_0\right)} \int_{\mathbb{R}^N}\left|\nabla V_{0, j}\right|^2dx
\end{equation*}
and
\begin{equation*}
  \sum_{j=1}^{J\left(\varepsilon_0\right)} \int_{\mathbb{R}^N}\left|\nabla V_{0, j}\right|^2dx \leq \int_{\mathbb{R}^N}\left|\nabla z_{0, n}\right|^2dx+\int_{\mathbb{R}^N}\left|\nabla z_{0, n}\right|^2dx+o(1) \leq 2 \int_{\mathbb{R}^N}|\nabla W|^2dx.
\end{equation*}
Let $M=C_0$ with $C_0$ as in \ref{eq8.8},
\begin{equation*}
  A=\widetilde{C} \int_{\mathbb{R}^N}|\nabla W|^2dx, A^{\prime}=A+\int_{\mathbb{R}^N}|\nabla W|^2dx, \varepsilon_0<\frac{\varepsilon_0(M, A, A^{\prime}, N)}{2},
\end{equation*}
where $\varepsilon_0(M, A, A^{\prime}, N)$ is defined as in Proposition \ref{P5.1}. Fix $\varepsilon_0$ and choose $n$ so large that $\|\nabla R_{n, \varepsilon_0}\|_{L_T^\infty H^1\cap L_T^2 L_x^{\frac{2 N}{N+2}}}<\varepsilon_0$ and so that all the above properties hold. Then Proposition \ref{P5.1} indicates that the conclusion is valid in the case when \eqref{eq8.1} holds.
\end{proof}
\begin{remark}\label{r8.1}
Assume that $\left\{z_{0, n}\right\}$ in Lemma \ref{L5.4} are all radial. Then $V_{0, j}, w_n$ can be chosen to be radial and we can choose $x_{j, n} \equiv 0$. This follows directly from Keraani's proof \cite{KSK2001}. If we then define $(\mathrm{SC})$ and $E_C$ by restricting only to radial functions, we obtain a $u_C$ as in Lemma \ref{L5.2} which is radial, and we can establish Lemma \ref{L5.3} with $x(t) \equiv 0$.
\end{remark}
\section{Rigidity theorem}
 In this section we will prove the following:
\begin{theorem}\label{t9.1}
Assume that $u_0 \in H^1$ is such that
$$
E(u_0)<E(W), \ \int_{\mathbb{R}^N}\left|\nabla u_0\right|^2dx<\int_{\mathbb{R}^N}|\nabla W|^2dx.
$$
Assume that $u$ be the solution of \eqref{eq1.1} and $\left.u\right|_{t=0}=u_0$ with maximal interval of existence $(-T_{-}(u_0), T_{+}(u_0))$. If there exists $\lambda(t)>0$, for $t \in\left[0, T_{+}\left(u_0\right)\right)$, with the property that
$$
K=\left\{v(x, t)=\frac{1}{\lambda(t)^{\frac{N-2}{2}}} u\left(\frac{x}{\lambda(t)}, t\right): t \in\left[0, T_{+}\left(u_0\right)\right)\right\}
$$
is such that $\overline{K}$ is compact in $H^1$. Then $T_{+}\left(u_0\right)=+\infty, u_0 \equiv 0$.
\end{theorem}
\begin{remark}\label{r9.1}
We conjecture that Theorem \ref{t9.1} remains true if $v(x, t)=$ $\frac{1}{\lambda(t)^{\frac{N-2}{2}}} u\left(\frac{x-x(t)}{\lambda(t)}, t\right)$, with $x(t) \in \mathbb{R}^N, t \in\left[0, T_{+}\left(u_0\right)\right)$. In other words, for ``energy subcritical'' initial data, compactness up to the invariances of the equation, for solutions, is only true for $u \equiv 0$.
\end{remark}
We start out with a special case of the strengthened form of Theorem \ref{t9.1}, i.e.,

\begin{lemma}\label{L9.1}
Assume that $u, v, \lambda(t), x(t)$ are as in Remark \ref{r9.1}, that $|x(t)| \leq C_0$ and that $\lambda(t) \geq A_0>0$. Then the conclusion of Theorem \ref{t9.1} holds. Moreover, if $T_{+}\left(u_0\right)<+\infty$, the hypothesis $|x(t)| \leq C_0$ is not needed.
\end{lemma}

In the next lemma we will collect some useful facts:
\begin{lemma}\label{L9.2}
Let $u, v$ be as in Lemma \ref{L9.1}.

i) Let $\delta_0>0$ be such that $E(u_0) \leq(1-\delta_0) E(W)$. Then for all $t \in$ $\left[0, T_{+}\left(u_0\right)\right)$, we have
$$
\begin{gathered}
\int_{\mathbb{R}^N}|\nabla u(t)|^2dx \leq(1-\delta_1) \int_{\mathbb{R}^N}|\nabla W|^2dx, \\
 \int_{\mathbb{R}^N}(|\nabla u|^2-F(u))dx \geq \bar{\delta} \int_{\mathbb{R}^N}|\nabla u|^2dx,\\
 \int_{\mathbb{R}^N}(|\nabla u|^2-f(u)\overline{u})dx \geq \bar{\delta} \int_{\mathbb{R}^N}|\nabla u|^2dx,\\
C_{1, \delta_0} \int_{\mathbb{R}^N}\left|\nabla u_0\right|^2dx \leq E(u_0) \leq C_2 \int_{\mathbb{R}^N}\left|\nabla u_0\right|^2dx, \\
E(u(t))=E\left(u_0\right), \\
C_{1, \delta_0} \int_{\mathbb{R}^N}\left|\nabla u_0\right|^2dx \leq \int_{\mathbb{R}^N}|\nabla u(t)|^2dx \leq C_2 \int_{\mathbb{R}^N}\left|\nabla u_0\right|^2dx .
\end{gathered}
$$

ii)
$$
\begin{aligned}
& \int_{\mathbb{R}^N}|\nabla v(t)|^2dx \leq C_2 \int_{\mathbb{R}^N}|\nabla W|^2dx, \\
& \|v(t)\|_{L_x^{2^*}}^2 \leq C_3 \int_{\mathbb{R}^N}|\nabla W|^2dx .
\end{aligned}
$$

iii) For all $x_0 \in \mathbb{R}^N$
$$
\int_{\mathbb{R}^N} \frac{|v(x, t)|^2}{\left|x-x_0\right|^2}dx \leq C_4 \int_{\mathbb{R}^N}|\nabla W|^2 dx.
$$

iv) For each $\varepsilon_0>0$, there exists $R(\varepsilon_0)>0$, such that, for $0 \leq t<T_{+}(u_0)$, we have
$$
\int_{|x| \geq R(\varepsilon_0)}\left(|\nabla v|^2dx+F(v)+\frac{|v|^2}{|x|^2}\right)dx \leq \varepsilon_0.
$$
\end{lemma}
\begin{proof}
Using Theorem \ref{t5.1}, Corollary \ref{c5.2} and Sobolev embedding, it is easy to see that i) and ii) hold. iii) follows from Hardy's inequality. using Sobolev embedding and the Hardy inequality, follows from the compactness of $ \overline{K}$.
\end{proof}
The next lemma is a localized virial identity about general nonlinear equation. The proof idea comes from Merle \cite{FM1992}.

\begin{lemma}\label{L9.3}
Let $\psi \in C_0^{\infty}(\mathbb{R}^N), t \in\left[0, T_{+}\left(u_0\right)\right)$. Then:

i)
$$
\frac{d}{d t} \int_{\mathbb{R}^N}|u|^2 \psi d x=2 \operatorname{Im} \int_{\mathbb{R}^N} \bar{u} \nabla u \nabla \psi d x
$$

ii)
\begin{equation*}
  \frac{d^2}{d t^2} \int_{\mathbb{R}^N}|u|^2 \psi d x=   4\int_{\mathbb{R}^N}\Delta \psi |\nabla u|^2 d x+\frac{1}{2} \int_{\mathbb{R}^N} \nabla|u|^2 \nabla(\Delta \psi)dx+\int_{\mathbb{R}^N} \Delta \psi [F(u)-f(u)\bar{u}]dx.
\end{equation*}
\end{lemma}
\begin{proof}
By \eqref{eq1.1} and direct calculation, we get
\begin{eqnarray*}
\frac{d}{d t} \int_{\mathbb{R}^N} \psi(x)|u(t, x)|^2 d x & =&2 \int_{\mathbb{R}^N} \operatorname{Re} \psi \frac{\partial u}{\partial t} \bar{u}dx \\
& =&2 \int_{\mathbb{R}^N} \operatorname{Re}\left[i \Delta u+if(u)\right] \bar{u} \psi dx \\
& =&-2 \operatorname{Im} \int_{\mathbb{R}^N} \Delta u \bar{u} \psi dx\\
&=&2 \operatorname{Im} \int_{\mathbb{R}^N} \nabla u \bar{u} \nabla \psi dx
\end{eqnarray*}
because of $\operatorname{Im} \int_{\mathbb{R}^N} \nabla u \overline{\nabla u}   \psi dx=0$ and part (i) follows.

Using (i), we know that
\begin{eqnarray*}
\frac{d^2}{d t^2} \int_{\mathbb{R}^N} \psi(x)|u(t, x)|^2 d x & =&2\left[\operatorname{Im} \int_{\mathbb{R}^N} \nabla \psi \bar{u} \nabla \frac{\partial u}{\partial t}dx+\operatorname{Im} \int_{\mathbb{R}^N} \nabla \psi \frac{\overline{\partial u}}{\partial t} \nabla udx\right] \\
& =&2\left[2 \operatorname{Im} \int_{\mathbb{R}^N} \nabla \psi \frac{\overline{\partial u}}{\partial t} \nabla udx-\operatorname{Im} \int_{\mathbb{R}^N} \Delta \psi \bar{u} \frac{\partial u}{\partial t}dx\right] .
\end{eqnarray*}
On the one hand,
\begin{eqnarray*}
-\operatorname{Im} \int_{\mathbb{R}^N} \Delta \psi \bar{u} \frac{\partial u}{\partial t}dx
& =&-\operatorname{Re} \int_{\mathbb{R}^N} \Delta \psi \bar{u}\left(\Delta u+f(u)\right)dx \\
& =&-\int_{\mathbb{R}^N} \Delta \psi f(u)\bar{u}dx+\int_{\mathbb{R}^N} \Delta \psi|\nabla u|^2dx+\frac{1}{2} \int_{\mathbb{R}^N} \nabla|u|^2 \nabla(\Delta \psi)dx.
\end{eqnarray*}
On the other hand, note that
\begin{eqnarray*}
  -2 \operatorname{Re} \int_{\mathbb{R}^N} \nabla \psi \Delta u \overline{\nabla u}dx &=& 2 \operatorname{Re} \int_{\mathbb{R}^N} \Delta \psi \nabla u \overline{\nabla u}dx+2 \operatorname{Re} \int_{\mathbb{R}^N} \nabla \psi  \overline{\Delta u} \nabla u dx \\
    &=& 2 \operatorname{Re} \int_{\mathbb{R}^N} \Delta \psi \nabla u \overline{\nabla u}dx+2 \operatorname{Re} \int_{\mathbb{R}^N} \nabla \psi   \overline{\Delta u \overline{\nabla u}} dx   \\
    &=& 2 \operatorname{Re} \int_{\mathbb{R}^N} \Delta \psi \nabla u \overline{\nabla u}dx+2 \operatorname{Re} \int_{\mathbb{R}^N} \nabla \psi \Delta u \overline{\nabla u}dx,
\end{eqnarray*}
so
\begin{equation*}
   -2 \operatorname{Re} \int_{\mathbb{R}^N} \nabla \psi \Delta u \overline{\nabla u}dx=  \operatorname{Re} \int_{\mathbb{R}^N} \Delta \psi \nabla u \overline{\nabla u}dx.
\end{equation*}
Therefore,  we have
\begin{eqnarray*}
2 \operatorname{Im} \int_{\mathbb{R}^N} \nabla \psi \frac{\overline{\partial u}}{\partial t} \nabla u dx&=&-2 \operatorname{Im} \int_{\mathbb{R}^N} \nabla \psi \frac{\partial u}{\partial t} \overline{\nabla u}dx\\
&= & -2 \operatorname{Re} \int_{\mathbb{R}^N} \nabla \psi \Delta u \overline{\nabla u}dx-2 \operatorname{Re} \int_{\mathbb{R}^N} \nabla \psi f(u) \overline{\nabla u}dx \\
&= & \operatorname{Re} \int_{\mathbb{R}^N} \Delta \psi \nabla u \overline{\nabla u}dx-2 \operatorname{Re} \int_{\mathbb{R}^N} \nabla \psi f(u) \overline{\nabla u}dx \\
&= &  \int_{\mathbb{R}^N}\Delta \psi |\nabla u|^2 d x + \int_{\mathbb{R}^N} \Delta \psi F(u)dx .
\end{eqnarray*}
Combining the above two estimates, we obtain
 \begin{eqnarray*}
&&\frac{d^2}{d t^2} \int_{\mathbb{R}^N} \psi(x)|u(t, x)|^2 d x \\
& =&2\left[2 \operatorname{Im} \int_{\mathbb{R}^N} \nabla \psi \frac{\overline{\partial u}}{\partial t} \nabla udx-\operatorname{Im} \int_{\mathbb{R}^N} \Delta \psi \bar{u} \frac{\partial u}{\partial t}dx\right]\\
& =&2\left[\int_{\mathbb{R}^N}\Delta \psi |\nabla u|^2 d x + \int_{\mathbb{R}^N} \Delta \psi F(u)dx -\int_{\mathbb{R}^N} \Delta \psi f(u)\bar{u}dx+\int_{\mathbb{R}^N} \Delta \psi|\nabla u|^2dx\right.\\
&&\left.+\frac{1}{2} \int_{\mathbb{R}^N} \nabla|u|^2 \nabla(\Delta \psi)dx\right]\\
& =&4\int_{\mathbb{R}^N}\Delta \psi |\nabla u|^2 d x+\frac{1}{2} \int_{\mathbb{R}^N} \nabla|u|^2 \nabla(\Delta \psi)dx+\int_{\mathbb{R}^N} \Delta \psi [F(u)-f(u)\bar{u}]dx,
\end{eqnarray*}
which implies that the conclusion is valid.
\end{proof}

\begin{proof}[\bf Proof of Lemma \ref{L9.1} ]
The proof splits in two cases, the finite time blow-up case for $u$ and the infinite time of existence for $u$.

\textbf{Case 1 } $T_{+}\left(u_0\right)<\infty$. Note first that $\lambda(t) \rightarrow \infty$ as $t \rightarrow T_{+}\left(u_0\right)$. If not, there exists $t_i \uparrow T_{+}\left(u_0\right)$ such that $\lambda\left(t_i\right) \rightarrow \lambda_0<+\infty$. Let
\begin{equation*}
  v_i(x)=\frac{1}{[\lambda\left(t_i\right)]^{\frac{N-2}{2}}} u\left(\frac{x-x\left(t_i\right)}{\lambda\left(t_i\right)}, t_i\right).
\end{equation*}
By the compactness of $\overline{K}$, there exists $v(x) \in H^1$  such that $v_i \rightarrow v$ in $H^1$. Hence, \begin{equation*}
  u\left(x-\frac{x\left(t_i\right)}{\lambda\left(t_i\right)}, t_i\right)=[\lambda\left(t_i\right)]^{\frac{N-2}{2}} v_i\left([\lambda\left(t_i\right)] x\right) \rightarrow \lambda_0^{\frac{N-2}{2}} v\left(\lambda_0 x\right) \text{ in } H^1
\end{equation*}
because of $\lambda\left(t_i\right) \geq A_0, \lambda_0 \geq A_0$. Let now $h(x, t)$ be the solution of \eqref{eq1.1} with data $\lambda_0^{\frac{N-2}{2}} v\left(\lambda_0 x\right)$ at time $T_{+}\left(u_0\right)$ and
\begin{equation*}
  \|h\|_{L_{\left(T_{+}\left(u_0\right)-\delta, T_{+}\left(u_0\right)+\delta\right)}^\frac{2(N+2)}{N-2}W^{1,{\frac{2N(N+2)}{N^2+4}}}}<+\infty
\end{equation*}
in an interval $\left(T_{+}\left(u_0\right)-\delta, T_{+}\left(u_0\right)+\delta\right)$. Let $h_i(x, t)$ be the solution with data at $T_{+}\left(u_0\right)$ equal to $u\left(x-\frac{x\left(t_i\right)}{\lambda\left(t_i\right)}, t_i\right)$. Then Lemmas \ref{L3.1} and \ref{L3.2} guarantee that
$$
\sup _i\left\|h_i\right\|_{L_{\left(T_{+}\left(u_0\right)-\delta, T_{+}\left(u_0\right)+\delta\right)}^\frac{2(N+2)}{N-2}W^{1,{\frac{2N(N+2)}{N^2+4}}}}<+\infty .
$$
But $u\left(x-\frac{x\left(t_i\right)}{\lambda\left(t_i\right)}, t+t_i-T_{+}\left(u_0\right)\right)=h_i(x, t)$, contradicting finite blow-up since $T_{+}\left(u_0\right)<\infty$.

Let us prove now a decay result for $u$ from the concentration properties in $L^{2^*}$ of $u$ at $T_{+}\left(u_0\right)$. Let us now fix $\varphi \in C_0^{\infty}(\mathbb{R}^N)$, $\varphi$ radial, $\varphi \equiv 1$ for $|x| \leq 1, \varphi \equiv 0$ for $|x| \geq 2$ and set $\varphi_R(x)=\varphi(\frac{x}{R})$. Define
$$
y_R(t)=\int_{\mathbb{R}^N}|u(x, t)|^2 \varphi_R(x) d x, \quad t \in\left[0, T_{+}\left(u_0\right)\right) .
$$
We then have
\begin{equation}\label{eq9.1}
  \left|y_R^{\prime}\right| \leq C_N \int_{\mathbb{R}^N}|\nabla W|^2dx .
\end{equation}
Indeed, it follows Lemma \ref{L9.2} and \ref{L9.3} that
\begin{eqnarray*}
  \left|y_R^{\prime}\right| & \leq& \frac{2}{R}\left|\operatorname{Im} \int_{\mathbb{R}^N} \bar{u} \nabla u \nabla \varphi(\frac{x}{R}) d x\right| \\
& \leq& C_N\left(\int_{\mathbb{R}^N}|\nabla u|^2dx\right)^{\frac{1}{2}} \cdot\left(\int_{\mathbb{R}^N} \frac{|u|^2}{|x|^2}dx\right)^{\frac{1}{2}} \leq C_N \int_{\mathbb{R}^N}|\nabla W|^2dx,
\end{eqnarray*}
We also have:
\begin{equation}\label{eq9.2}
  \text{For all } R>0,  \int_{|x|<R}  |u(x, t)|^2 d x \rightarrow 0 \text{ as } t \rightarrow T_{+}\left(u_0\right).
\end{equation}
In fact, $u(y, t)=[\lambda(t)]^{\frac{N-2}{2}} v(\lambda(t) y+x(t), t)$, so that
\begin{eqnarray*}
&&\int_{|x|<R}|u(x, t)|^2 d x\\
&= & [\lambda(t)]^{-2} \int_{|y|<R \lambda(t)}|v(y+x(t), t)|^2 d y \\
&= & [\lambda(t)]^{-2} \int_{B(x(t), R \lambda(t))}|v(z, t)|^2 d z \\
&= & [\lambda(t)]^{-2} \int_{B(x(t), R \lambda(t)) \cap B(0, \varepsilon R \lambda(t))}|v(z, t)|^2 d z  +[\lambda(t)]^{-2} \int_{B(x(t), R \lambda(t)) \backslash B(0, \varepsilon R \lambda(t))}|v(z, t)|^2 d z \\
&= & A+B .
\end{eqnarray*}
By H\"older inequality$(\frac{2}{2^*}+\frac{2}{N}=1)$,
\begin{equation*}
  A  \leq  [\lambda(t)]^{-2}(\varepsilon R \lambda(t))^{N\cdot\frac{2}{N}}\|v\|_{L^{2^*}}^2 \leq \varepsilon^2 R^2 C_3 \int_{\mathbb{R}^N}|\nabla W|^2dx,
\end{equation*}
which is small with $\varepsilon$. Now, fixed $\varepsilon>0$, note that $\lambda(t) \rightarrow+\infty$ as $t \rightarrow T_{+}(u_0)$, by Lemma \ref{L9.2} and H\"older inequality, we have
\begin{equation*}
  B   \leq [\lambda(t)]^{-2}(R \lambda(t))^{N \frac{2}{N}}\|v\|_{L^{2^*}(|x| \geq \varepsilon R \lambda(t))}^2
  =R^2\|v\|_{L^{2^*}(|x| \geq \varepsilon R \lambda(t))}^2 \rightarrow 0 \text { as } t \rightarrow T_{+}(u_0),
\end{equation*}

According to \eqref{eq9.1} and \eqref{eq9.2}, we have
$$
y_R(0) \leq y_R(T_{+}(u_0))+C_N T_{+}(u_0) \int_{\mathbb{R}^N}|\nabla W|^2dx=C_N T_{+}\left(u_0\right) \int_{\mathbb{R}^N}|\nabla W|^2dx.
$$
Hence, letting $R \rightarrow+\infty$ we obtain $u_0 \in L^2(\mathbb{R}^N)$. Arguing as before,
\begin{equation*}
  \left|y_R(t)-y_R\left(T_{+}\left(u_0\right)\right)\right| \leq C_N\left(T_{+}\left(u_0\right)-t\right) \int_{\mathbb{R}^N}|\nabla W|^2dx,
\end{equation*}
so that
\begin{equation*}
  y_R(t) \leq C_N\left(T_{+}\left(u_0\right)-t\right) \int_{\mathbb{R}^N}|\nabla W|^2dx.
\end{equation*}
Letting $R \rightarrow \infty$, we see that
\begin{equation*}
  \|u(t)\|_{L^2}^2 \leq C_N\left(T_{+}\left(u_0\right)-t\right) \int_{\mathbb{R}^N}|\nabla W|^2dx
\end{equation*}
and so by the conservation of the $L^2$ norm $\left\|u(T_{+}(u_0))\right\|_{L^2}=\left\|u_0\right\|_{L^2}=0$. But then $u \equiv 0$, contradicting $T_{+}(u_0)<\infty$.

\textbf{Case 2 } $T_{+}\left(u_0\right)=+\infty$.
In this case we assume, in addition, that $|x(t)| \leq C_0$. We first note that
for each $\varepsilon>0$, there exists $R(\varepsilon)>0$ such that, for all $t \in[0, \infty)$, it holds
\begin{equation}\label{eq9.3}
  \int_{|x|>R(\varepsilon)}\left(|\nabla u|^2+F(u)+\frac{|u|^2}{|x|^2}\right)dx \leq \varepsilon ,
\end{equation}
\begin{equation}\label{eq9.4}
  \int_{|x|>R(\varepsilon)}\left(|\nabla u|^2+f(u)\overline{u}+\frac{|u|^2}{|x|^2}\right)dx \leq \varepsilon ,
\end{equation}
In fact, $u(y, t)=[\lambda(t)]^{\frac{N-2}{2}} v(\lambda(t) y+x(t), t)$, so that
\begin{eqnarray*}
\int_{|y|>R(\varepsilon)}|\nabla u(y, t)|^2 d y & =&\int_{|y|>R(\varepsilon)} [\lambda(t)]^N|\nabla v(\lambda(t) y+x(t), t)|^2 d y \\
& =&\int_{|z|>R(\varepsilon) \lambda(t)}|\nabla v(z+x(t), t)|^2 d z \\
& \leq& \int_{|z| \geq R(\varepsilon) A_0}|\nabla v(z+x(t), t)|^2 d z \\
& \leq& \int_{|\alpha| \geq R(\varepsilon) A_0-C_0}|\nabla v(\alpha, t)|^2 d \alpha
\end{eqnarray*}
and the statement for this term now follows from Lemma \ref{L9.2} $iv)$. The other terms are handled similarly. There exists $R_0>0$ such that, for all $t \in[0,+\infty)$, we have
\begin{equation}\label{eq9.5}
  8 \int_{|x| \leq R_0}|\nabla u|^2dx-8 \int_{|x| \leq R_0}F(u)dx \geq C_{\delta_0} \int_{\mathbb{R}^N}\left|\nabla u_0\right|^2dx .
\end{equation}
\begin{equation}\label{eq9.6}
  8 \int_{|x| \leq R_0}|\nabla u|^2dx-8 \int_{|x| \leq R_0}f(u)\overline{u}dx \geq C_{\delta_0} \int_{\mathbb{R}^N}\left|\nabla u_0\right|^2dx .
\end{equation}
In fact, \eqref{eq5.12} combined with Lemma \ref{L9.2} i) yields
\begin{equation*}
  8 \int_{\mathbb{R}^N}|\nabla u|^2dx-8 \int_{\mathbb{R}^N}F(u)dx \geq \widetilde{C}_{\delta_0} \int_{\mathbb{R}^N}\left|\nabla u_0\right|^2dx.
\end{equation*}
Now combine this with \eqref{eq9.3}, with $\varepsilon=\varepsilon_0 \int_{\mathbb{R}^N}\left|\nabla u_0\right|^2dx$ to obtain \eqref{eq9.5}. Similarly, it can be proven that \eqref{eq9.6} is established.

To prove Case 2, we choose $\varphi \in C_0^{\infty}(\mathbb{R}^N)$, radial, with $\varphi(x)=|x|^2$ for $|x| \leq 1, \varphi(x) \equiv 0$ for $|x| \geq 2$. Define
\begin{equation*}
  z_R(t)=\int_{\mathbb{R}^N}|u(x, t)|^2 R^2 \varphi(\frac{x}{R}) d x,
\end{equation*}
then we have
\begin{eqnarray*}
 \left|z_R^{\prime}(t)\right| \leq C_{N, \delta_0} \int_{\mathbb{R}^N}\left|\nabla u_0\right|^2 R^2dx\text { for } t>0, \\
 z_R^{\prime \prime} \geq C_{N, \delta_0} \int_{\mathbb{R}^N}\left|\nabla u_0\right|^2dx\text { for } R \text { large enough, } t>0.
\end{eqnarray*}
In fact, from Lemma \ref{L9.3} i),
\begin{eqnarray*}
\left|z_R^{\prime}(t)\right| & \leq& 2 R\left|\operatorname{Im} \int_{\mathbb{R}^N}  \overline{ u} \nabla u \nabla \varphi(\frac{x}{R}) d x\right| \leq C_N R \int_{0 \leq x \mid \leq 2 R} \frac{|x|}{|x|}|\nabla u||u|dx \\
& \leq& C_N R^2\left(\int_{\mathbb{R}^N}|\nabla u|^2dx\right)^{\frac{1}{2}}\left(\int_{\mathbb{R}^N} \frac{|u|^2}{|x|^2}dx\right)^{\frac{1}{2}} \\
&\leq& C_N R^2 \int_{\mathbb{R}^N}\left|\nabla u_0\right|^2dx
\end{eqnarray*}
because of Lemma \ref{L9.2} i). In view of \eqref{eq9.3}, \eqref{eq9.4}, \eqref{eq9.5}, \eqref{eq9.6}  and Lemma \ref{L9.3}, ii), we have
\begin{eqnarray*}
z_R^{\prime \prime}(t)&=& 4\int_{\mathbb{R}^N}\Delta \varphi |\nabla u|^2 d x+\frac{1}{2} \int_{\mathbb{R}^N} \nabla|u|^2 \nabla(\Delta \varphi)dx+\int_{\mathbb{R}^N} \Delta \varphi [F(u)-f(u)\bar{u}]dx \\
&\geq& 4\int_{\mathbb{R}^N}[  |\nabla u|^2-F(u) ]d x +4\int_{\mathbb{R}^N}   [|\nabla u|^2-f(u)\bar{u}]dx \\
& \geq&  4\int_{|x|\leq R}[  |\nabla u|^2-F(u) ]d x +4\int_{|x|\leq R}   [|\nabla u|^2-f(u)\bar{u}]dx   \\
&&-C_N \int_{R \leq|x| \leq 2 R}\left[|\nabla u|^2+\frac{|u|^2}{|x|^2}+|u|^{2^*}\right]dx\\
&\geq& C_{N, \delta_0} \int_{\mathbb{R}^N}\left|\nabla u_0\right|^2dx
\end{eqnarray*}
for $R$ large. If we now integrate in $t$, we have $z_R^{\prime}(t)-z_R^{\prime}(0) \geq C_{N, \delta_0} t \int_{\mathbb{R}^N}\left|\nabla u_0\right|^2dx$, but we also have $\left|z_R^{\prime}(t)-z_R^{\prime}(0)\right| \leq 2 C_N R^2 \int_{\mathbb{R}^N}\left|\nabla u_0\right|^2dx$, a contradiction for $t$ large, unless $\int_{\mathbb{R}^N}\left|\nabla u_0\right|^2dx=0$.
\end{proof}
\begin{proof}[\bf Proof of Theorem \ref{t9.1} ]
Assume that $u_0 \not \equiv 0$ so that $\int_{\mathbb{R}^N}\left|\nabla u_0\right|^2dx>0$ and because of Lemma \ref{L9.2} i) (which is still valid here), $E(u_0) \geq C_{1, \delta_0} \int_{\mathbb{R}^N}\left|\nabla u_0\right|^2dx$ and hence $E(u_0)>0$. Because of Lemma \ref{L9.1}, we only need to treat the case where there exists $\left\{t_n\right\}_{n=1}^{\infty}, t_n \geq 0, t_n \uparrow T_{+}\left(u_0\right)$, so that
$$
\lambda\left(t_n\right) \rightarrow 0 .
$$
(If $t_n \rightarrow t_0 \in\left[0, T_{+}\left(u_0\right)\right.)$, we obtain for all $R>0, \int_{|x| \geq R}\left|v(t_0)\right|^{2 *}dx=0$ but $\left.\int_{\mathbb{R}^N}\left|\nabla v\left(t_0\right)\right|^2dx>0\right)$. After possibly redefining $\left\{t_n\right\}_{n=1}^{\infty}$ we can assume that
$$
\lambda(t_n) \leq 2 \inf_{t \in\left[0, t_n\right]} \lambda(t)
$$
and from our hypothesis
$$
w_n(x)=\frac{1}{\lambda\left(t_n\right)^{\frac{N-2}{2}}} u\left(\frac{x}{\lambda\left(t_n\right)}, t_n\right) \rightarrow w_0 \text { in } H^1 .
$$
By Theorem \ref{t5.1} we have
\begin{equation*}
  E(W)>E(w_0)=E(u_0)>0, \int_{\mathbb{R}^N}|\nabla u(t)|^2dx \leq  (1-\delta_1) \int_{\mathbb{R}^N}|\nabla W|^2dx
\end{equation*}
so that $\int_{\mathbb{R}^N}\left|\nabla w_0\right|^2dx<\int_{\mathbb{R}^N}|\nabla W|^2dx$. Thus $w_0 \not \equiv 0$. Let us now consider solutions of \eqref{eq1.1}, $w_n(x, \tau), w_0(x, \tau)$ with data $w_n(-, 0), w_0(-, 0)$ at $\tau=0$, defined in maximal intervals $\tau \in\left(-T_{-}\left(w_n\right), 0\right], \tau \in\left(-T_{-}\left(w_0\right), 0\right]$ respectively. Since $w_n \rightarrow w_0$ in $H^1$, $\varliminf\limits_{n \rightarrow \infty} T_{-}\left(w_n\right) \geq T_{-}\left(w_0\right)$ and
for each $\tau \in\left(-T_{-}\left(w_0\right), 0\right], w_n(x, \tau) \rightarrow w_0(x, \tau)$ in $H^1$.

Note that by uniqueness in \eqref{eq1.1}, for $0 \leq t_n+\frac{\tau}{[\lambda\left(t_n\right)]^2}$, $w_n(x, \tau)=\frac{1}{[\lambda\left(t_n\right)]^{\frac{N-2}{2}}} u(\frac{x}{[\lambda\left(t_n\right)]}, t_n+\frac{\tau}{[\lambda\left(t_n\right)]^2})$. Remark that $\varliminf\limits_{n \rightarrow \infty}  \tau_n= \varliminf\limits_{n \rightarrow \infty}  t_n [\lambda\left(t_n\right)]^2 \geq T_{-}\left(w_0\right)$ and thus for all $\tau \in\left(-T_{-}\left(w_0\right), 0\right]$ for $n$ large,
$0 \leq t_n+\frac{\tau}{\lambda\left(t_n\right)^2} \leq t_n$. Indeed, if $\tau_n \rightarrow \tau_0<T_{-}\left(w_0\right)$, then $w_n\left(x,-\tau_n\right)=$ $\frac{1}{[\lambda\left(t_n\right)]^{\frac{N-2}{2}}} u_0\left(\frac{x}{\lambda\left(t_n\right)}\right) \rightarrow w_0\left(x,-\tau_0\right)$ in $H^1$ with $\lambda\left(t_n\right) \rightarrow 0$ which is a contradiction from $u_0 \not \equiv 0, w_0 \not \equiv 0$.

Fix now $\tau \in\left(-T_{-}\left(w_0\right), 0\right]$, for $n$ sufficiently large $v\left(x, t_n+\frac{\tau}{\lambda\left(t_n\right)^2}\right)$, $\lambda\left(t_n+\frac{\tau}{ \lambda\left(t_n\right)^2}\right)$ are defined and we have
\begin{eqnarray*}
&& v\left(x, t_n+\frac{\tau}{\lambda\left(t_n\right)^2}\right) \\
& =&\frac{1}{\lambda\left(t_n+\frac{\tau}{\lambda\left(t_n\right)^2}\right)^{\frac{N-2}{2}}} u\left(\frac{x}{\lambda\left(t_n+\frac{\tau}{ \lambda\left(t_n\right)^2}\right)}, t_n+\frac{\tau}{ \lambda\left(t_n\right)^2}\right) \\
& =&\frac{1}{\widetilde{\lambda}_n(\tau)^{\frac{N-2}{2}}} w_n\left(\frac{x}{\widetilde{\lambda}_n(\tau)}, \tau\right),
\end{eqnarray*}
where
$$
\widetilde{\lambda}_n(\tau)=\frac{\lambda\left(t_n+\frac{\tau}{ \lambda\left(t_n\right)^2}\right)}{\lambda\left(t_n\right)} \geq \frac{1}{2}
$$
because of the fact $\lambda(t_n) \leq 2 \inf\limits_{t \in\left[0, t_n\right]} \lambda(t))$. One can assume after passing to a subsequence that $\widetilde{\lambda}_n\left(t_n+\frac{\tau}{ \lambda\left(t_n\right)^2}\right) \rightarrow \widetilde{\lambda}_0(\tau)$ with $\frac{1}{2} \leq \widetilde{\lambda}_0(\tau) \leq+\infty$ and $v\left(x, t_n+\frac{\tau}{ \lambda\left(t_n\right)^2}\right) \rightarrow v_0(x, \tau)$ in $H^1$, as $n \rightarrow \infty$. Remark that $\widetilde{\lambda}_0(\tau)< +\infty$. If not, we will have
\begin{equation*}
  \frac{1}{\widetilde{\lambda}_n(\tau)^{\frac{N-2}{2}}} w_0\left(\frac{x}{\widetilde{\lambda}_n(\tau)}, \tau\right) \rightarrow v_0(x, \tau),
\end{equation*}
which implies $w_0(x, \tau)=0$ which contradicts $E\left(w_0\right)=E\left(u_0\right)>0$. Thus $\widetilde{\lambda}_0(\tau)<+\infty$ and $v_0(x, \tau)=\frac{1}{\widetilde{\lambda}_0(\tau)^{\frac{N-2}{2}}} w_0\left(\frac{x}{\widetilde{\lambda}_0(\tau)}, \tau\right)$ where $v_0(\tau) \in \overline{K}$. We thus obtain a contradiction from Lemma \ref{L9.1}.
\end{proof}
\begin{corollary}\label{c9.1}
Assume that $E\left(u_0\right)<E(W), \int_{\mathbb{R}^N}\left|\nabla u_0\right|^2dx<\int_{\mathbb{R}^N}|\nabla W|^2dx$ and $u_0$ is radial. Then the solution $u$ of the Cauchy problem \eqref{eq1.1} with data $u_0$ at $t=0$ has time interval of existence $I=(-\infty,+\infty),\|u\|_{L_I^\frac{2(N+2)}{N-2}W^{1,{\frac{2N(N+2)}{N^2+4}}}}<+\infty$ and there exists $u_{0,+}, u_{0,-}$ in $H^1$ such that
$$
\lim\limits_{t \rightarrow+\infty}\left\|u(t)-e^{i t \Delta} u_{0,+}\right\|_{H^1}=0, \  \lim\limits_{t \rightarrow-\infty}\left\|u(t)-e^{i t \Delta} u_{0,-}\right\|_{H^1}=0 .
$$
Moreover, if we define $\delta_0$ so that $E(u_0) \leq(1-\delta_0) E(W)$, there exists a function $M(\delta_0)$ so that
$$
\|u\|_{L_I^\frac{2(N+2)}{N-2}W^{1,{\frac{2N(N+2)}{N^2+4}}}} \leq M\left(\delta_0\right) .
$$
\end{corollary}
\begin{proof}
By the integral equation in Lemma \ref{L3.1}, we know that $u(t)$ is radial for each $t \in I$. Using Remark \ref{r8.1} and Theorem \ref{t9.1} we obtain $I=(-\infty,+\infty),\|u\|_{L_I^\frac{2(N+2)}{N-2}W^{1,{\frac{2N(N+2)}{N^2+4}}}}<+\infty$. Now Remark \ref{r4.1} finishes the proof of the first statement.

For the last statement, let
$$
\begin{gathered}
D_{\delta_0}=\left\{u_0 \in H^1 \text { radial, } \int_{\mathbb{R}^N}\left|\nabla u_0\right|^2dx<\int_{\mathbb{R}^N}|\nabla W|^2dx \text { and } E\left(u_0\right) \leq\left(1-\delta_0\right) E(W)\right\} \\
M\left(\delta_0\right)=\sup _{u \in D_{\delta_0}}\|u\|_{L_I^\frac{2(N+2)}{N-2}W^{1,{\frac{2N(N+2)}{N^2+4}}}} .
\end{gathered}
$$
We need to prove $M(\delta_0)<+\infty$. If not there is a sequence $u_{0, n}$ in $D_{\delta_0}$ and the corresponding solutions $u_n$ such that $\left\|u_n\right\|_{L_I^\frac{2(N+2)}{N-2}W^{1,{\frac{2N(N+2)}{N^2+4}}}} \rightarrow+\infty$ as $n \rightarrow+\infty$. Note that we can assume that $\left\|e^{i t \Delta} u_{0, n}\right\|_{L_I^\frac{2(N+2)}{N-2}W^{1,{\frac{2N(N+2)}{N^2+4}}}} \geq \rho$, with $\rho$ as in Lemma \ref{L3.1}. Arguing as in the proof of Lemma \ref{L5.2}, we would conclude that first $J=1$ in the decomposition given in Lemma \ref{L5.4} and then since $\left\|U_1\right\|_{L_I^\frac{2(N+2)}{N-2}W^{1,{\frac{2N(N+2)}{N^2+4}}}}<+\infty$ we reach a contradiction.
\end{proof}


\begin{thebibliography}{99}

\bibitem{HBGP2022} H. Bahouri, G. Perelman, Global well-posedness for the derivative nonlinear Schr\"odinger equation. Invent. Math., 229(2)(2022), 639--688.



\bibitem{HB1983}  H. Berestycki, P. L. Lions,  Nonlinear scalar field equations, I. existence of a ground state, Arch. Ration. Mech. Anal., 82(1983), 313--345.

\bibitem{HBII1983}  H. Berestycki, P. L. Lions, Nonlinear scalar field equations, II. Existence of infinitely many solutions, Arch. Rational Mech. Anal, 82(4)(1983), 347--375.

\bibitem{JBJL1976}  J. Bergh, J. Lofstrom, Interpolation Spaces. An Introduction. Grundlehren der Mathematischen Wissenschaften, No. 223. Berlin, New York: Springer 1976.

\bibitem{TBDH2016} T. Boulenger, D. Himmelsbach, E. Lenzmann, Blow up for fractional NLS, J. Funct. Anal, 271(2016), 2569--2603.

\bibitem{JBAB2014} J. Bourgain, A. Bulut,  Almost sure global well posedness for the radial nonlinear Schr\"odinger equation on the unit ball I: the 2D case, Annales de l'Institut Henri Poincar$\mathrm{\acute{e}}$ C, Analyse non lin$\mathrm{\acute{e}}$aire. No longer published by Elsevier, 2014, 31(6)(2014), 1267--1288.



\bibitem{JB1999} J. Bourgain, Global well-posedness of defocusing critical nonlinear Schr\"odinger equation in the radial case, J. Am. Math. Soc., 12(1999), 145--171.

\bibitem{JBM1999} J. Bourgain, New global well-posedness results for nonlinear Schr\"odinger equations, AMS Colloquium Publications, 46, 1999.

\bibitem{TCS2003} T. Cazenave, Semilinear Schr\"odinger Equations. Courant Lecture Notes in Mathematics, vol. 10. New York: New York University Courant Institute of Mathematical Sciences 2003.

\bibitem{TCS1993} T. Cazenave, An Introduction to Nonlinear Schr\"odinger Equations, Text. Metod. Mat., vol. 26, Univ. Fed. Rio de Janeiro, 1993.

\bibitem{TC1990} T. Cazenave, F. B. Weissler, The Cauchy problem for the critical nonlinear Schr\"odinger equation in $H^s$, Nonlinear Anal., Theory Methods Appl., 14(1990), 807--836.

\bibitem{VDD2018}  V. D. Dinh, Well-posedness of nonlinear fractional Schr\"odinger and wave equations in Sobolev spaces, Int. J. Appl. Math., 31 (2018), 483--525.





\bibitem{JGGV1979} J. Ginibre, G. Velo, On a class of nonlinear Schr\"odinger equations, I, The Cauchy problem, general case, J. Funct. Anal., 32(1)(1979), 1--32.

\bibitem{JGGV1984}  J. Ginibre and G. Velo, Scattering theory in the energy space for a class of nonlinear Schr\"odinger equations, J. Math. Pures Appl., 64(1984), 363--401.

\bibitem{DHNI2024} D. Hennig, N. I. Karachalios, D. Mantzavinos, J. Cuevas-Maraver, I. G. Stratis, On the proximity between the wave dynamics of the integrable focusing nonlinear Schr\"odinger equation and its non-integrable generalizations, J. Differ. Equ., 397(2024), 106-165.



\bibitem{MKTT1998} M. Keel, T. Tao, Endpoint Strichartz estimates, Am. J. Math., 120(1998), 955--980.

\bibitem{KCEMF2006} C. E. Kenig, F. Merle, Global well-posedness, scattering and blow-up for the energy-critical, focusing, non-linear Schr\"odinger equation in the radial case, Invent. Math., 166(3)(2006), 645--675.




\bibitem{KSK2001} S. Keraani, On the defect of compactness for the Strichartz estimates of the Schr\"odinger equations, J. Differ. Equations, 175(2001), 353--392.


\bibitem{KSK2006} S. Keraani, On the blow up phenomenon of the critical Schr\"odinger equation, J. Funct. Anal., 235(2006), 171--192.

\bibitem{FM1992} F. Merle, On uniqueness and continuation properties after blow-up time of self-similar solutions of nonlinear schr\"odinger equation with critical exponent and critical mass, Comm. Pure Appl. Math, 45(2)(1992), 203-254.

\bibitem{TOYW2020} T. Oh, Y. Wang, Global well-posedness of the one-dimensional cubic nonlinear Schr\"odinger equation in almost critical spaces, J. Differ. Equ., 269(1)(2020), 612--640.



\bibitem{TTMV2007} T. Tao, M. Visan, X. Zhang,  Global well-posedness and scattering for the defocusing mass-critical nonlinear Schr\"odinger equation for radial data in high dimensions, Duke Math. J., 140(1)(2007), 165--202.

 

\bibitem{KY1987} K. Yajima, Existence of solutions for Schr\"odinger evolution equations, Comm. Math. Phys., 110(1987), 415--426.


\bibitem{XYHY2024} X. Yu, H. Yue, Z. Zhao, Global Well-posedness and scattering for fourth-order Schr\"odinger equations on waveguide manifolds, SIAM J. Math. A., 56(1)(2024), 1427--1458.








\end{thebibliography}
\end{document}